\documentclass[a4paper,10pt]{siamonline0516}
\usepackage[english]{babel}
\usepackage{amsmath,amssymb}           
\usepackage{graphicx}        
\usepackage{epsfig,epstopdf}   

\newtheorem{thm}{Theorem}
\newtheorem{lem}[thm]{Lemma}

\newtheorem{rem}[thm]{Remark}
\newtheorem{defn}[thm]{Definition}

\newcommand{\R}{{\mathbb R}}

\newcommand{\argmin}{\mathop{\rm argmin}}%

\newcommand{\cO}{{\mathcal O}}

\newcommand{\cN}{{\mathcal N}}

\newcommand{\bE}{{\mathbf E}}

\newcommand{\eps}{\varepsilon}
\newcommand{\KL}{{\it KL}}

\newcommand{\wrt}{with respect to }

\begin{document}

\title{Stochastic gradient descent and fast relaxation to thermodynamic equilibrium:\\ a stochastic control approach}

\author{Tobias Breiten, Carsten Hartmann{\thanks{Institut f\"ur Mathematik, BTU Cottbus-Senftenberg, Cottbus, Germany, e-mail: \email{carsten.hartmann@b-tu.de}}}, Lara Neureither, Upanshu Sharma}

\date{\today}

\maketitle

\begin{abstract}

We study the convergence to equilibrium of an underdamped Langevin equation that is controlled by a linear feedback force. Specifically, we are interested in sampling the possibly multimodal invariant probability distribution of a Langevin system at small noise (or low temperature), for which the dynamics can easily get trapped inside metastable subsets of the phase space. We follow [Chen et al., J. Math. Phys. {\bf 56}, 113302, 2015] and consider a Langevin equation that is simulated at a high temperature, with the control playing the role of a friction that balances the additional noise so as to restore the original invariant measure at a lower temperature. We discuss different limits as the temperature ratio goes to infinity and prove convergence to a limit dynamics. It turns out that, depending on whether the lower (``target'') or the higher (``simulation'') temperature is fixed, the controlled dynamics converges either to the overdamped Langevin equation or to a deterministic gradient flow. This implies that (a) the ergodic limit and the large temperature separation limit do not commute in general, and that (b) it is not possible to accelerate the speed of convergence to the ergodic limit by making the temperature separation larger and larger. We discuss the implications of these observation from the perspective of stochastic optimisation algorithms and enhanced sampling schemes in molecular dynamics.        

\end{abstract}

\section{Introduction}

Methods to accelerate the convergence of Monte Carlo algorithms that sample the equilibrium distribution of a possibly high-dimensional system have been around since at least the invention of the Metropolis algorithm \cite{Robert2018}. Yet, they have recently enjoyed an enormous boost in attention, partly driven by the latest developments in machine learning \cite{Ma2019,ZhangRethink2017} and theoretical advances in hypocoercive systems \cite{Kontis2016,Villani2009}. Depending on the scientific community, accelerated sampling algorithms come under various different names, e.g.\ enhanced sampling  \cite{enhancedSampling}, smart sampling \cite{smartSampling}, guided sampling \cite{guidedSampling}, to mention just a few. They come into play when the target probability distribution is not just  high-dimensional, but also multimodal, which is a challenge for any Monte Carlo algorithm \cite{mengersen1996}. Similar issues arise in machine learning and the (stochastic) optimisation of high-dimensional, non-convex functions, and it has been argued that Monte Carlo sampling can improve the training of large-scale   networks by reducing overfitting and accelerating convergence \cite{adLaLa2019,Ma2019}.    

A model for Monte Carlo simulations that is popular in statistical mechanics and that shares features of Brownian dynamics and ballistic motion dynamics is the underdamped Langevin model \cite{Betancourt2018}. The underdamped Langevin dynamics is a second-order stochastic differential equation (SDE) and is based on Hamilton's equations perturbed by friction and noise. Besides the physical appeal of the model, it has been conjectured that the underdamped Langevin dynamics shows superior convergence over its first-order counterpart, the so-called \emph{overdamped} Langevin equation, and this conjecture has been recently proved using the 2-Wasserstein distance \cite{Cheng2018}. The fact that the (underdamped) Langevin equation is a second-order SDE makes it harder to analyse from a theoretical point of view, because the underlying infinitesimal generator is degenerate and non-symmetric. However, it is precisely the degeneracy and the lack of symmetry that offers the possibility to tweak the dynamics by changing the drift in a way that the convergence to equilibrium is accelerated without altering the invariant measure \cite{Duncan2017,Eberle2019}; see also \cite{Hwang2005,Lelievre2013,ReyBellet2015} for related ideas. 

In view of the aforementioned, two recent developments are relevant: sampling methods based on either optimal control or mean field approaches \cite{Hu2019,Hu2020}, and the application of Monte Carlo methods from statistical mechanics---especially molecular dynamics---to problems in machine learning or Bayesian inference  \cite{Girolami2011,Neal2012}. 
Some of the modern stochastic optimisation methods from machine learning, like ADAM, AdaGrad or RMSProp adaptively control the learning rate so as to improve the convergence to a local minimum, but they also share many features with adaptive versions of the Langevin equation \cite{Goodfellow2015,adLaLa2019}. 
It therefore comes as no surprise that the Langevin model is the basis for the stochastic modified equation approach that can be used to analyse state of the art momentum-based stochastic optimisation algorithms like ADAM \cite{An2019,SME2019}.
The common feature of all these algorithms is that they are designed to adaptively sample only  relevant parts of a complicated energy landscape or loss function in order to accelerate convergence to equilibrium (i.e.~to an equilibrium distribution or an approximation of the global optimum). 
Control theory plays a key role here in that many of the aforementioned adaptive algorithms can be interpreted and analysed as solutions to stochastic optimal control problems \cite{Ferre2018,Hartmann2019,IHP2019}; cf.~also \cite{Mitter2003}.\\ 

In this paper, we take up an idea that was put forward in \cite{Pavon2015} and that is reminiscent of the Schr\"odinger bridge approach (e.g.~\cite{daipra1991,Reich2019}). Inspired by recent ideas \cite{Pavon2016a,Pavon2016b} for solving the Schr\"odinger bridge problem for a linear SDE, the idea here  is to augment the (nonlinear) Langevin equation by a linear feedback control that plays the role of bridging the equilibrium distribution of an uncontrolled Langevin equation at a high (``simulation'') temperature to a non-equilibrium target probability measure at a lower (``target'') temperature on an infinite time horizon. The control target (i.e.~the non-equilibrium invariant measure) is reached asymptotically in the long time limit by letting the control act as a friction force that balances the additional noise in the high-temperature Langevin equation so as to restore the invariant measure at the target temperature. The approach rests on standard arguments from equilibrium statistical mechanics (specifically: fluctuation-dissipation relations), combined with functional inequalities (specifically: FIR inequalities) from which an ergodic optimal control problem is derived. The latter leads to the minimum dissipation approach of~\cite{Pavon2015}.  

Adopting a slightly different perspective, the control approach can be thought of as a non-asymptotic version of simulated annealing, in which the control target is to efficiently sample the modes of a probability density without getting trapped in local modes too often. The question of finding good temperature protocols in simulated annealing is intimately related to the question of accelerating the speed of convergence towards the stationary distribution by optimising the spectral gap of the underlying generator \cite{Holley1989,Loewe1996}. Having these two sides of a coin, it is thus natural to look at the limit as the ratio between the simulation temperature and the target temperature goes to infinity. It turns out that, depending on whether the target or the simulation temperature is fixed, the controlled dynamics converges either to the overdamped Langevin equation or to a deterministic gradient flow. This implies that the ergodic limit and the large temperature separation limit do not commute in general. As a consequence, the controlled Langevin dynamics interpolates between sampling and stochastic gradient descent (under an appropriate temperature scaling), and this fact can be used to devise algorithms for non-convex optimisation problems. On the other hand, it implies that it is not possible to accelerate the system indefinitely by increasing the simulation temperature under the constraint of asymptotically reaching the stationary distribution, which is in agreement with related  results in \cite{Duncan2017,Eberle2019}.

The article is outlined as follows. In Section \ref{sec:langevin}, we introduce the  Langevin model with linear feedback controls and derive optimality conditions for the target invariant measure. The large temperature limits of the model are studied in Section \ref{sec:limits}. All theoretical findings are discussed and illustrated with numerical examples in Section \ref{sec:numerics}. We conclude with a short discussion in Section \ref{sec:discussion}. The article contains an appendix that records various technical identities, derivations and proofs.

\section{High-temperature Langevin dynamics}\label{sec:langevin}

Let $(X,Y)=(X_{t},Y_{t})_{t\ge 0}$ be the solution of the controlled Langevin equation
\begin{equation}\label{langevinSDE}
\begin{aligned}
dX_{t} & = Y_{t}\,dt\,,\quad X_{0}=x\\
dY_{t} & = \left(\sigma u_{t} -\nabla V(X_{t}) - \gamma Y_{t} \right)dt + \sigma dW_{t}\,,\quad Y_{0}=y\,
\end{aligned}
\end{equation}
on the phase space $\R^{n}\times\R^{n}$. Here $V\colon\R^{n}\to\R$ is a smooth potential energy function that is bounded from below and coercive, $W_t$ is a standard Brownian motion in $\R^n$,  $\gamma\in\R^{n\times n}$ is a symmetric positive definite matrix, and the noise coefficient $\sigma\in\R^{n\times n}$ satisfies the fluctuation-dissipation relation
\begin{equation}\label{FDR}
2\gamma = \bar{\beta}\sigma\sigma^{\top}\,
\end{equation}
for some $\bar{\beta}>0$. The control $u$ assumes values in $\R^{n}$ and will be specified below; for $u\equiv 0$ the dynamics admits the unique invariant measure  (omitting normalisation constants here and in what follows)
\begin{equation}\label{invMeasure}
\bar{\mu}_{\infty}(dx,dy) = \exp\left(-\bar{\beta} H(x,y)\right)dx dy\,,
\end{equation}
with the Hamiltonian
\begin{equation}\label{H}
H\colon\R^{n}\times\R^{n}\to\R\,, \quad H(x,y) = \frac{1}{2}|y|^{2} + V(x)\,.
\end{equation}

We call $1/\bar{\beta}$ the ``simulation temperature''. Our aim is to simulate (\ref{langevinSDE}) at inverse temperature $\bar{\beta}$, and choose a control so that the dynamics converges to a probability distribution $\mu_\infty$ at a lower temperature $1/\beta$ that we call the ``target temperature''. We follow the approach proposed in \cite{Pavon2015}, and consider controls of the form $u_{t}=B^{\top}Y_{t}$ for some suitable matrix $B\in\R^{n\times n}$. The following Lemma characterises the admissible matrices $B$ that achieve the aforementioned aim. 

In what follows the operators $\nabla_x,\nabla_y$ are the gradients and $\nabla^2_y$ is the Hessian with respect to the variable in the subscript, where $x,y\in\mathbb R^n$. For two matrices $D,C$, we use $D\colon C=\rm{tr}(D^\top C)$. 
\begin{lem}\label{lem:invMeasure2}
Let $\beta, \bar{\beta}>0$ and $u_{t}=B^{\top}Y_{t}$ for some $B\in\R^{n\times n}$. Then $\mu_{\infty}$ with density $\rho_{\infty}\propto \exp(-\beta H)$ is invariant under (\ref{langevinSDE}) if and only if  
\begin{equation}\label{FDR2}
B\sigma^{\top} + \sigma B^{\top} = (\bar{\beta}- \beta)\sigma\sigma^{\top} \,.
\end{equation}

\end{lem}
\begin{proof} 
The infinitesimal generator of (\ref{langevinSDE}) on $C^{2}$ has the form $L=L_{1}+L_{2}$, with 
\begin{equation}
L _{1} = y\cdot \nabla_{x} - \nabla_{x}V\cdot \nabla_{y} \quad L _{2} =  \frac{1}{2}\sigma\sigma^{\top}\colon \nabla^{2}_{y} + \left(\sigma B^{\top}y -\gamma y\right)\cdot \nabla_{y}\,,
\end{equation}
where the Hamiltonian part, $L_{1}$, is skew-symmetric \wrt the standard $L^{2}$ scalar product. Hence it suffices to show that, for any smooth and integrable function $f$,
\begin{equation}\label{Lfdmu}
\int (L_{2} f)d\mu_{\infty} = 0
\end{equation}
if and only if (\ref{FDR2}) holds. Using integration by parts, we obtain 
\begin{align*}
\int &(L_{2} f)d\mu_{\infty}  = \int f\Bigl(  \frac{1}{2}\sigma\sigma^{\top}\colon \nabla^{2}_{y} - \left(\sigma B^{\top}y -\gamma y\right)\cdot \nabla_{y} - {\rm tr}(\sigma B^{\top}-\gamma)\Bigr)\!\rho_{\infty} \,dxdy\\
& = \int f\Bigl(  \frac{1}{2}\sigma\sigma^{\top}\colon (\beta^{2}yy^{\top} -\beta I_{n\times n}) + \beta\left(\sigma B^{\top}y -\gamma y\right)\cdot y - {\rm tr}(\sigma B^{\top}-\gamma)\Bigr)\!\rho_{\infty}\,dxdy\\
& = \int f(x,y) g(y) \rho_{\infty}(x,y)\,dxdy\,,
\end{align*}
with 
\begin{equation*}
g(y) =   \frac{1}{2}\sigma\sigma^{\top}\colon (\beta^{2}yy^{\top} -\beta I_{n\times n}) + \beta\bigl(\sigma B^{\top}y -\gamma y\bigr)\cdot y - {\rm tr}(\sigma B^{\top}-\gamma)\,.
\end{equation*}
Since $\rho_{\infty}>0$, it follows by (\ref{FDR})  that $g=0$ if and only if the symmetric part of the matrix $\sigma B^{T}$ is equal to $\frac{1}{2}  (\bar{\beta}- \beta)\sigma\sigma^{\top} $ or, in other words, 
\begin{equation*}
\int (L_{2} f)d\mu_{\infty}  = 0 \quad\forall f\in C^{2}\cap L^{1} \quad \Longleftrightarrow \quad g=0  \quad \Longleftrightarrow \quad B\sigma^{\top} + \sigma B^{\top} = (\bar{\beta}- \beta)\sigma\sigma^{\top}  \,.
\end{equation*}
Hence the assertion is proved. 
\end{proof}
Equation (\ref{FDR2}) is a $T$-Sylvester equation for which existence and uniqueness results can be found in \cite{Wim94}. Unlike the standard Lyapunov equation for a symmetric unknown, its solution need not be unique. In fact, note that  $B=\frac{1}{2}(\bar{\beta}-\beta)\sigma+\sigma M$ is a solution to (\ref{FDR2}) for any skew-symmetric matrix $M=-M^\top$. However, since $\sigma \sigma^\top $ is positive definite and, hence, $\sigma$ is invertible there is also a unique solution such that $\sigma B^\top$ is symmetric. This particular solution corresponds to $M=0$ and will be of interest in the next section. 

Roughly speaking, the matrix $\sigma B^{\top}$ plays the role of a friction coefficient that can cause positive or negative dissipation, depending on whether $\beta>\bar{\beta}$ or $\beta\le \bar{\beta}$.  Typically, we want to speed up the convergence to the target distribution $\mu_\infty$ and so the system is simulated at a temperature $1/\bar{\beta}$ that is higher than the target temperature $1/\beta$, i.e.~we have $\beta>\bar{\beta}$. Then, the eigenvalues of $B\sigma^\top+\sigma B^{\top}$ have strictly negative real part and therefore the control term $\sigma u$ dissipates energy. If $\beta=\bar{\beta}$, equation (\ref{FDR2}) implies that $\sigma B^{\top}$ must be skew-symmetric.

We will henceforth refer to $u$ as either ``control'', ``friction'' or ``dissipation''.

\subsection{A stochastic control problem}

There are many matrices $B$ satisfying the relation  (\ref{FDR2}). For example, if  $\beta=\bar{\beta}$, any matrix $B$ with the property $B\sigma^{\top}=-\sigma B^\top$ preserves the invariant measure $\mu_\infty=\bar{\mu}_\infty$, especially the zero matrix $B=0$. In fact, the symmetric part of $B\sigma^\top$ is unique, whereas the anti-symmetric or skew-symmetric part is not, which is in accordance with previous works on non-reversible perturbations of reversible diffusions, e.g.~\cite{Hwang2005}

We now identify a unique control that has a property that it minimises the relative entropy between the controlled dynamics and an uncontrolled Langevin system that is initialised at the correct equilibrium distribution, with the aim of accelerating the convergence to equilibrium. The stochastic control approach leads to a Langevin dynamics that has a minimum-dissipation property, in that it is as close as possible to a reversible system in a wide sense (i.e.\ including momentum flips \cite{Duncan2017}), since the friction term becomes symmetric again. 

Letting $\rho_t$ denote the probability density function of the marginal law of the controlled process (\ref{langevinSDE}) at time $t\ge 0$. Then, by Lemma \ref{thm:FIR} in Appendix \ref{sec:FIR}, it follows that 
\begin{equation}\label{FIR0}
	\KL(\rho_\infty|\rho_T) -  	\KL(\rho_\infty|\rho_0) \le  \int_0^T R(u_{t})\,dt\,,
\end{equation}
where $\KL(\cdot,\cdot)$ as defined in (\ref{KL}) denotes the Kullback-Leibler divergence (or relative entropy) between two probability densities, and 
\begin{equation}\label{rate0}
	R(u) = \frac{1}{2}\bE\!\left[\Bigl|\sigma u - \frac{1}{2}(\bar{\beta}-\beta)\sigma\sigma^{\top}y\Bigr|_{(\sigma\sigma^{\top})^{-1}}^{2}\right]\,
\end{equation}
is the asymptotic entropy production (AEP) rate when the uncontrolled process is started in equilibrium; see Appendix \ref{sec:AEP} for details \cite{Roeckner2016,Sharma2016}. Here, $|z|_g=z^\top g z$ for a symmetric and positive definite matrix $g$, and   $\bE[\cdot]$ denotes the expectation with respect to the target density $\rho_\infty$.
Our aim is to minimise the AEP rate $R$ over the  set of admissible controls
\begin{equation}\label{finalDensity}
\mathcal{U}=\bigl\{u\in C([0,\infty),\R^n)\colon u_t=B^\top Y_t,\, B\sigma^{\top} + \sigma B^{\top} = (\bar{\beta}- \beta)\sigma\sigma^{\top} \bigr\}\,,
\end{equation}
which boils down to a minimisation over matrices $B\in\R^{n\times n}$ that are of the form 
\[
B=\frac{1}{2}(\bar{\beta}-\beta)\sigma+\sigma M\,.
\]
for some skew-symmetric matrix $M=-M^\top$. The AEP functional $\tilde{R}(B)=R(B^Ty)$ is strictly convex with the unique minimum $\tilde{R}(B_*)=0$ that is attained at 
\begin{equation}\label{optB}
	B_*=\frac{1}{2}(\bar{\beta}  - \beta)\sigma\,. 
\end{equation}
The optimal matrix $B_*$ has been computed in \cite{Pavon2015} by minimising a slightly different functional, however. 
Combining equation \eqref{optB} with (\ref{FDR}), we see that the controlled system that is simulated at temperature $1/\bar{\beta}$ satisfies a fluctuation-dissipation relation with parameter $\beta$: 
\begin{equation}\label{FDR3}
	2(\gamma-\sigma B_*^{\top}) = \beta\sigma\sigma^\top\,.
\end{equation}

	
	

\section{Large temperature separation limit}\label{sec:limits}
We are interested in the limit $\beta/\bar{\beta}\to\infty$, which can either correspond to the high simulation temperature limit $\bar{\beta}\to 0$ when the target temperature $1/\beta$ is held fixed, or to the zero target temperature limit $\beta\to\infty$ when the simulation temperature $1/\bar{\beta}$ is fixed.  

Using the relation~\eqref{FDR}, it is convenient to write (\ref{langevinSDE}) under the constraint (\ref{optB}) as 
\begin{align}\label{langevinSDEalt}
\begin{aligned}
dX_{t} & = Y_{t}\,dt\,,\quad X_{0}=x\\
dY_{t} & = \left(D^{\top} Y_{t} - \nabla V(X_{t}) \right)dt + \sigma dW_{t}\,,\quad Y_{0}=y\,
\end{aligned}
\end{align}
where 
\[
D^{\top} = - \frac{\beta}{\bar{\beta}}\, \gamma = D\,.
\]
Note that the scaled friction matrix $D$ satisfies the fluctuation-dissipation relation at the correct target temperature $\beta$, 
\begin{align*}
2D=-\beta\sigma\sigma^\top\quad \Longleftrightarrow\quad 2\gamma=\bar{\beta}\sigma\sigma^\top\,.
\end{align*}

\subsection*{Fixed simulation temperature}

We introduce the dimensionless parameter $\eps=\bar{\beta}/\beta\in(0,1]$. Upon rescaling time according to $t\mapsto t/\eps$, we can recast our Langevin equation (\ref{langevinSDEalt}) as 
\begin{subequations}\label{langevinSDEalt2}
\begin{align}
dX_{t} & = \frac{1}{\eps}Y_{t}\,dt\,,\quad X_{0}=x\label{langevinSDEalt2X}\\
dY_{t} & = - \frac{1}{\eps^2}\gamma Y_tdt - \frac{1}{\eps}\nabla V(X_{t})dt   + \frac{1}{\sqrt{\eps}}\sigma dW_{t}\,,\quad Y_{0}=y\,.\label{langevinSDEalt2Y}
\end{align}
\end{subequations}
We are interested in the limit $\eps\to 0$ when both $\sigma$ and $\gamma$ are held fixed. It turns out that, in this limit, the $x$-component of (\ref{langevinSDEalt2}) mimics a stochastic gradient descent algorithm for small $\eps$ in the sense that, as  $\eps\to 0$, it converges to the solution of the gradient system 
\begin{equation}\label{SGD}
\frac{d}{dt}x(t) = -\gamma^{-1}\nabla V(x(t))\,,\quad x(0)=x\,.
\end{equation} 
This can be seen as follows. Integrating the second equation  of (\ref{langevinSDEalt2}) from zero to $t$ yields
\begin{align}\label{Yint}
Y_{t} - y =  - \frac{1}{\eps^2}\gamma \int_{0}^{t}Y_s\,ds - \frac{1}{\eps}\int_{0}^{t}\nabla V(X_{s})\,ds   + \frac{1}{\sqrt{\eps}}\sigma W_{t}\,,
\end{align}
which, together with the first equation 
\begin{align}\label{Xint}
X_{t} - x = \frac{1}{\eps}\int_{0}^{t} Y_{s}\,ds \,,
\end{align}
leads to
\begin{align*}
X_{t} - x =  - \gamma^{-1}\left(\int_{0}^{t}\nabla V(X_{s})\,ds  - \sqrt{\eps}\sigma W_{t} - \eps(y - Y_{t})\right)\,,
\end{align*}
and therefore formally passing $\eps\rightarrow 0 $ we arrive at the limit equation. 
We now give the precise statement. 
\begin{theorem}\label{thm:SGD}
	Let $\nabla V$ be Lipschitz continuous on $\R^{n}$ with constant $L_V$. Furthermore let $X=X^{\eps}$ and $x$ be the solutions of  (\ref{langevinSDEalt2}) and (\ref{SGD}) with $X_{0}=x(0)$. Then, for any $T>0$,
	\begin{align*}
	\lim\limits_{\eps \to 0}\sup_{0\le t\le T}|X_t - x_t| = 0 \;\text{ almost surely.}
	\end{align*}
	Moreover we have the pre-asymoptotic bound  
	\begin{align*}
	\sqrt{ \bE\Bigl[\sup_{0\le t\le T}|X_t - x_t|^2\Bigr]} \le C\sqrt{\eps}\,,
	\end{align*}
	 for a constant $C=C(T,L_V,\gamma,\sigma,n)$ independent of $\eps$.  
\end{theorem}
The proof is based on~\cite[Prop. 2.14]{mathias2010free} and we outline the main steps in Appendix~\ref{app:SGD}. 

\subsection*{Fixed target temperature} By construction, under the limit  $\beta/\bar{\beta}\to\infty$ for fixed $\bar{\beta}$ the target measure is changed. If, however $\beta/\bar{\beta}\to\infty$ while $\beta$ is held fixed, then the target measure is obviously unaltered. We define $\varsigma:=\sigma\sqrt{\bar{\beta}/\beta}$, with 
\[
\beta\varsigma \varsigma^\top =\bar{\beta}\sigma \sigma^\top=2\gamma
\]
such that $\varsigma$ is independent of $\bar{\beta}$.
After rescaling time again according to $t\mapsto t/\eps$, we obtain 
\begin{equation}\label{langevinSDEalt3}
\begin{aligned}
dX_{t}  &= \frac{1}{\eps}Y_{t}\,dt\,,\quad X_{0}=x \\
dY_{t}  &= - \frac{1}{\eps^2}\gamma Y_tdt - \frac{1}{\eps}\nabla V(X_{t})dt   + \frac{1}{\eps}\varsigma\, dW_{t}\,,\quad Y_{0}=y. \end{aligned}
\end{equation}
Note the $1/\eps$ scaling of the noise term that is different from the $1/\sqrt{\eps}$ scaling in (\ref{langevinSDEalt2}).
We are interested in the limit $\eps\to 0$ when both $\varsigma$ and $\gamma$ are held fixed. (Note that the noise coefficient $\varsigma$ is independent of $\bar{\beta}$.) It is well-known that, under this scaling, the $x$-component converges pathwise to the solution of the overdamped Langevin equation 
\begin{equation}\label{Smoluchowski}
\gamma d\bar{X}_{t} = -\nabla V(\bar{X}_{t})dt + \varsigma\, dW_{t}\,,\quad \bar{X}_{0}=x\,.
\end{equation} 
The following result is due to Nelson  \cite[Thm. 10.1]{nelson1967}.
\begin{theorem}[High-friction limit]\label{thm:nelson}
	Let $\nabla V$ be Lipschitz continuous on $\R^{n}$. Further let $X=X^{\eps}$ and $\bar{X}$ be the solutions of  (\ref{langevinSDEalt3}) and (\ref{Smoluchowski}) respectively with $X_{0}=\bar{X}_{0}$. Then for any $T>0$, 
\begin{align*}
\lim\limits_{\eps\rightarrow 0}\sup_{0\le t\le T}|X_{t} - \bar{X}_{t}| = 0\, \;\text{ almost surely.}
\end{align*}
\end{theorem}

\section{Numerical examples}\label{sec:numerics}

We consider three different examples, with increasing degree of complexity: a linear system where we can explicitly control the speed of convergence, a high-dimensional bistable system, for which we want to compute a slowly converging marginal, and a Lennard-Jones cluster for which we seek a minimum energy configuration. 

\subsection{Optimal temperature ratio}

We consider a linear Langevin system of the form
\begin{equation}\label{linearSDE}
dZ_{t} = AZ_{t}dt + C dW_{t}
\end{equation}
with the matrices 
\begin{equation}\label{linearSDEcoeff1}
A = \begin{pmatrix}
0 & I_{n}\\ -K & \sigma B^\top - \gamma
\end{pmatrix}\in\R^{2n\times 2n}\,,\quad C =   \begin{pmatrix}
0\\  \sigma\end{pmatrix}\in\R^{2n\times n}\,,
\end{equation}
where $\gamma,K\in\R^{n\times n}$ are symmetric positive definite matrices. Using (\ref{FDR}) together with (\ref{optB}), the coefficient matrices can be recast as 
\begin{equation}\label{linearSDEcoeff2}
A = \begin{pmatrix}
0 & I_{n}\\ -K & - \frac{\beta}{\bar\beta}\gamma 
\end{pmatrix}\,, \quad C =   \begin{pmatrix}
0\\  \sqrt{2\bar{\beta}^{-1}}\,\gamma^{1/2} \end{pmatrix}\,.
\end{equation}
Note that (\ref{linearSDE}) is basically the controlled Langevin equation (\ref{langevinSDEalt}) with a quadratic Hamiltonian. Hence, by construction, it has a unique Gaussian invariant measure $\mu_{\infty}=\cN(0,\beta^{-1}\Sigma)$ with $\Sigma={\rm diag}(K^{-1},I_{n\times n})\in\R^{2n\times 2n}$.

We are interested in the optimal temperature ratio $\beta/\bar{\beta}$ that maximizes the speed of convergence to $\mu_{\infty}$. Using 
 \cite[Theorem 4.9]{ArnoldErb} and the Czisza\`r-Kullback-Pinsker inequality, the convergence is exponential with a rate given by the real part of the principal eigenvalue of $A$ (spectral abscissa).  
Specifically, letting $-r=\max\{{\rm Re}(\lambda)\colon Av=\lambda v\}$ and calling $\eta_{t}$ the probability density associated with the law of $Z_{t}$ at time $t>0$, we have, for some $C>0$, 
\begin{equation}\label{ArnoldErb}
\KL(\eta_{t}|\rho_{\infty}) \le C\exp(-2rt)
\end{equation}
where $\rho_{\infty}=\eta_{\infty}$ denotes the density associated with the invariant measure $\mu_{\infty}$ and $\KL(\eta|\rho)$ denotes the Kullback-Leibler (KL) divergence or relative entropy between two probability measures with densities $\eta$ and $\rho$; see Definition \ref{defn:kl} below. 

We want to determine the inverse temperature ratio $\beta/\bar{\beta}$ that brings the principal eigenvalue of $A$ as far away as possible from  the imaginary axis, i.e., we want to minimize the largest real part of the eigenvalues. To begin with, we consider the case $n=1$.  The two eigenvalues $A$ are given by
\begin{align*}
\lambda_{\pm} = \frac{{\rm tr}(A)}{2} \pm\sqrt{\Bigl(\frac{{\rm tr}(A)}{2}\Bigr)^{2} -  \det(A)}\,,
\end{align*}
and they are either negative real or form a conjugate complex pair, since ${\rm tr}(A)<0$ and
\begin{align*}
{\rm Re}\Biggl(\sqrt{\Bigl(\frac{{\rm tr}(A)}{2}\Bigr)^{2}  -  \det(A)}\Biggr) < \frac{{\rm tr}(A)}{2}.
\end{align*}
The optimal rate $2r^{*}={\rm tr}(A)$ is therefore achieved when the two eigenvalues are equal and real, which is exactly the case when $({\rm tr}(A))^{2}=4\det(A)$ or, equivalently, 
\begin{equation}\label{optRatio}
\frac{\beta}{\bar\beta} = 2\frac{\sqrt{K}}{\gamma}\,.
\end{equation}
This is to say that, for given $K$, the optimal ratio between simulation temperature and target temperature scales as $2/\gamma$ with the friction coefficient $\gamma$; for example, when  $K=\gamma=1$, the optimal simulation temperature is twice the target temperature.  Note, however, that the optimal temparature ratio can be smaller than 1 if the friction coefficient is large.  
\begin{figure}
	\centering
	\includegraphics[width=0.49\textwidth]{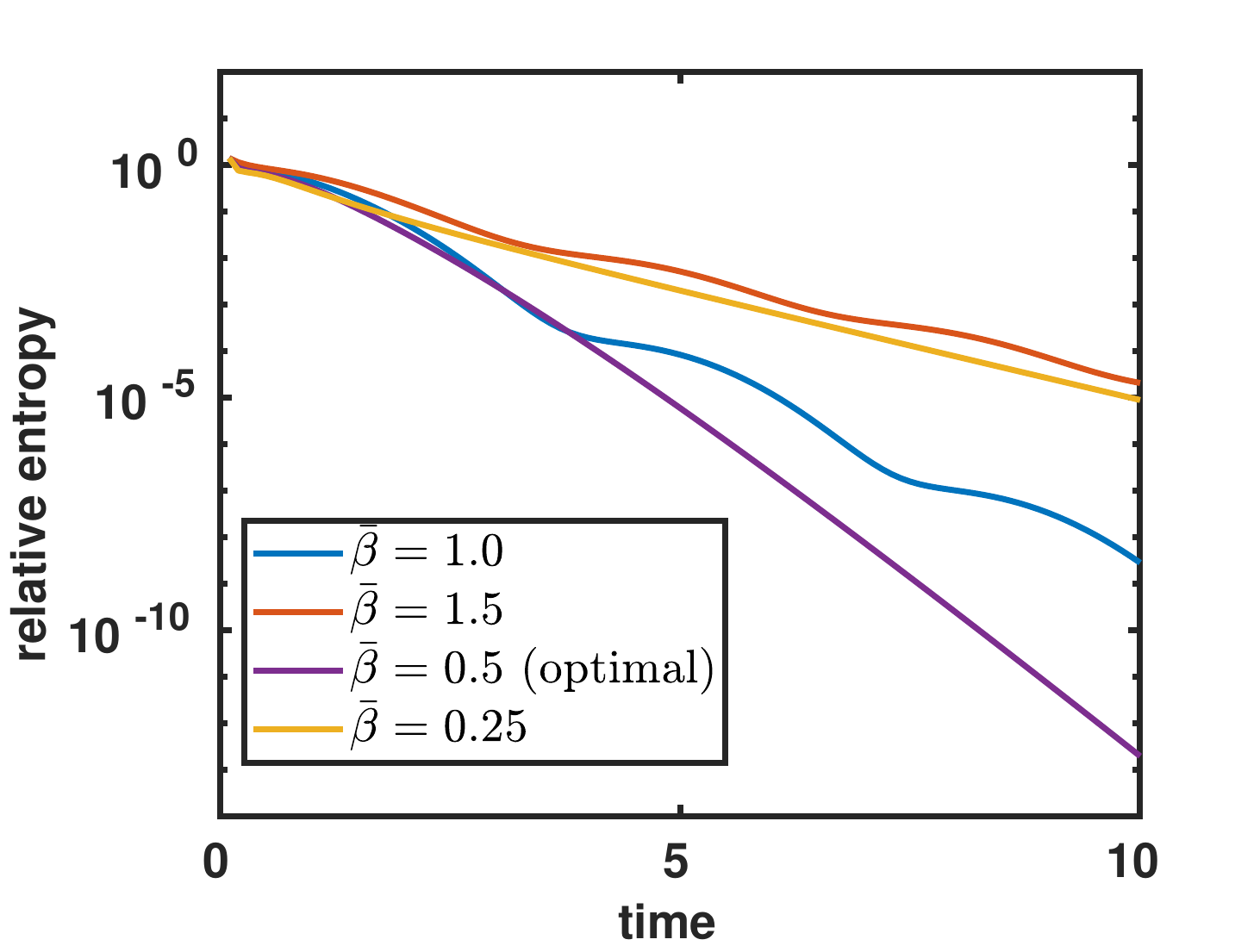}
	\includegraphics[width=0.49\textwidth]{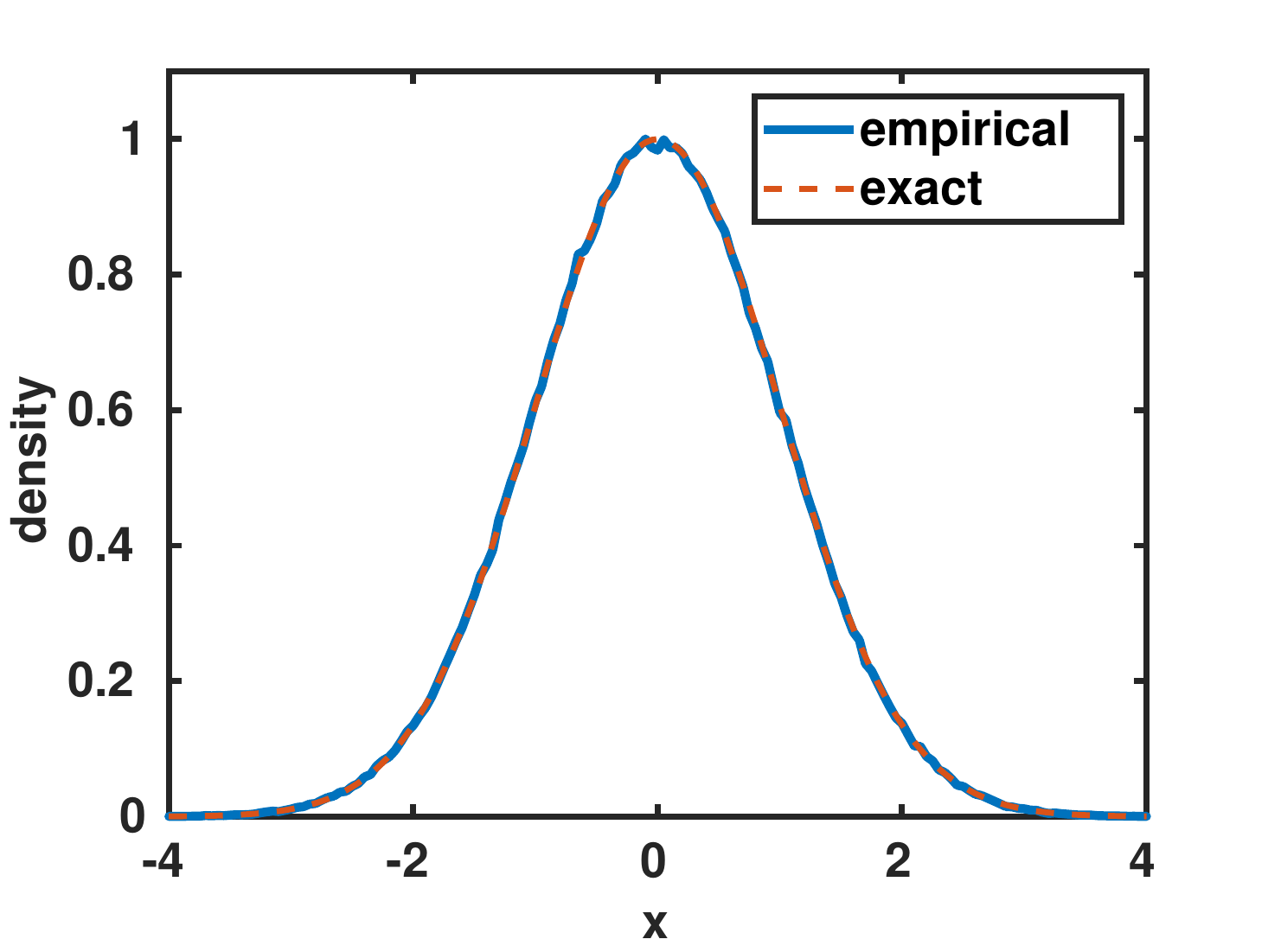}
	\caption{Relative entropy (left panel) between the time-dependent probability density of the Ornstein-Uhlenbeck process (\ref{linearSDE}) for various simulation temperatures $\bar{\beta}$. Since the initial data $Z_0$ satisfy  $\bE[Z_0]=0$, the plots show only the relaxation of the covariance $\Sigma_t=\bE[Z_tZ_t^T]$ to  $\Sigma_\infty$. For illustration, the right panel shows the exact stationary density and its Monte-Carlo approximation based on an ergodic average for a moderate value of $T$ and the optimal simulation temperature $\bar{\beta}=0.5$.}\label{fig:linearKL}
\end{figure}
\paragraph*{Illustration}
We illustrate the previous considerations with (\ref{linearSDE}) where we choose the coefficients
\begin{align*}
A = \begin{pmatrix}
0 & 1\\ -1 & - \beta/\bar{\beta}
\end{pmatrix}\,, \quad C =   \begin{pmatrix}
0\\  \sqrt{2\bar{\beta}} \end{pmatrix}\,
\end{align*}
for $\beta=1$ and random initial conditions $Z_{0}$, with 
\begin{align*}
\bE[Z_{0}]=0\,,\quad \bE[Z_{0}Z_{0}^{T}] = \begin{pmatrix}
 0.1 & 0\\ 0 & 0.1
\end{pmatrix}\,.
\end{align*}
We call $\Sigma_{0} =  \bE[Z_{0}Z_{0}^{T}]$. Then $\eta_{t}\sim\cN(0,\Sigma_{t})$ with 
\begin{align*}
\Sigma_{t} = \Sigma_\infty + \exp(At)(\Sigma_0 - \Sigma_\infty)\exp(A^Tt)
\end{align*}
The KL divergence between the two centered Gaussians $\eta_{t}$ and $\rho_{\infty}$ is readily seen to be~\cite{matrixCookbook}
\begin{align*}
\KL(\eta_t|\rho_t) = \frac{1}{2}\left({\rm trace}(\Sigma_t\Sigma_\infty^{-1}) - \log\left(\det(\Sigma_t)\det(\Sigma_\infty^{-1})\right) - 1\right)\,.
\end{align*}
Figure \ref{fig:linearKL} shows the results of a comparison for different values of $\bar{\beta}$. It is clearly visible that the optimal value $\bar{\beta}=\beta/2$ beats the equilibrium simulation $\bar{\beta}=\beta$ by several orders of magnitude (purple curve vs. blue curve) whereas the deviation between the distributions is clearly larger for other suboptimal temperature ratios (light blue and orange curves). Furthermore, a larger simulation temperature need not necessarily increase the speed of convergence as shown by the orange curve.

\paragraph*{Multidimensional linear system}
We briefly discuss a generalisation of the previous considerations to the case $n>1$.
Generally, since the eigenvalues of $A$ are either real or occur in complex conjugate pairs, the optimal eigenvalue bound and thus the maximally achievable rate of convergence in (\ref{ArnoldErb}) is when all eigenvalues have the same real part, i.e.\ when 
\begin{align*}
2r^{*} =  - \frac{{\rm tr}(A)}{n}\,.
\end{align*}
To see this we decompose the matrix $A\in\R^{2n\times 2n}$ in (\ref{linearSDEcoeff2}) into symmetric and antisymmetric parts according to 
\begin{align*}
    A = S + U\,, \quad S = \frac{1}{2}(A+A^T)\,,\; U = \frac{1}{2}(A-A^T)\,.
\end{align*}
The fact that ${\rm tr}(A)$ is a real number now implies that 
\begin{align*}
    {\rm tr}(A) = {\rm tr}(S) = \sum_{j=1}^k m_j \lambda_j = \sum_{j=1}^k m_j {\rm Re}(\lambda_j)
\end{align*}
where $m_j$ denotes the algebraic multiplicity of the eigenvalue $\lambda_j$. By construction, the sum of $m_j$ equals $2n$, and so
\begin{align*}
    {\rm tr}(A) \le \sum_{j=1}^k m_j \max\{{\rm Re}(\lambda_j):j=1,\ldots,k\} = - 2n r\,.
\end{align*}
Note, however, that this bound does not imply that the temperature ratio can be made arbitrarily large so as to improve the rate of convergence indefinitely. By Theorem \ref{thm:nelson}, the convergence rate for infinitely large temperature ratios equals the convergence rate of the overdamped Langevin equation at the target temperature. 
Moreover, the optimal convergence rate cannot always be reached. 
For example, let $A\in\R^{4\times 4}$ be defined as in (\ref{linearSDEcoeff2}), with   
\begin{align*}
 K  = I_2\,,\quad \gamma = \begin{pmatrix}
1 & 0\\ 0 & 2\end{pmatrix}\,.
\end{align*}
Introducing the shorthand $\alpha=\beta/\bar{\beta}$, the eigenvalues of the matrix are
\begin{equation}\label{linearSDE4x4}
    \lambda_{1,2} = -\frac{\alpha}{2} \pm\sqrt{\frac{\alpha^2}{4}-1}\,,\quad    
    \lambda_{3,4} = -\alpha \pm\sqrt{\alpha^2-1}\,.
\end{equation}
The optimal temperature ratio $\alpha$ that minimises the real part of the maximum eigenvalue is $\alpha^*=\sqrt{4/3}\approx 1.1547$, and it leads to the eigenvalues  (see Figure \ref{fig:4x4EWP})
\begin{align*}
    \lambda_{1,2}\approx -0.5774\pm 0.8165 i\,,\quad \lambda_3 \approx -0.5774\,,\quad \lambda_4 \approx -1.7321\,.
\end{align*}
Clearly, for any given simulation temperature (i.e. when $\alpha$ is constant), the spectral abscissa cannot be smaller than $r^*={\rm tr}(A)/4\approx -0.8660$, which determines the theoretically optimal speed of convergence (see page \pageref{optRatio}). We observe that the optimal rate of convergence is not reached in this case. It is  unclear, under which conditions this rate can be generally reached, and we refer to \cite{Vanetal09} for numerical algorithms for spectral abscissa minimization.
\begin{figure}
\centering
	\includegraphics[width=0.6\textwidth]{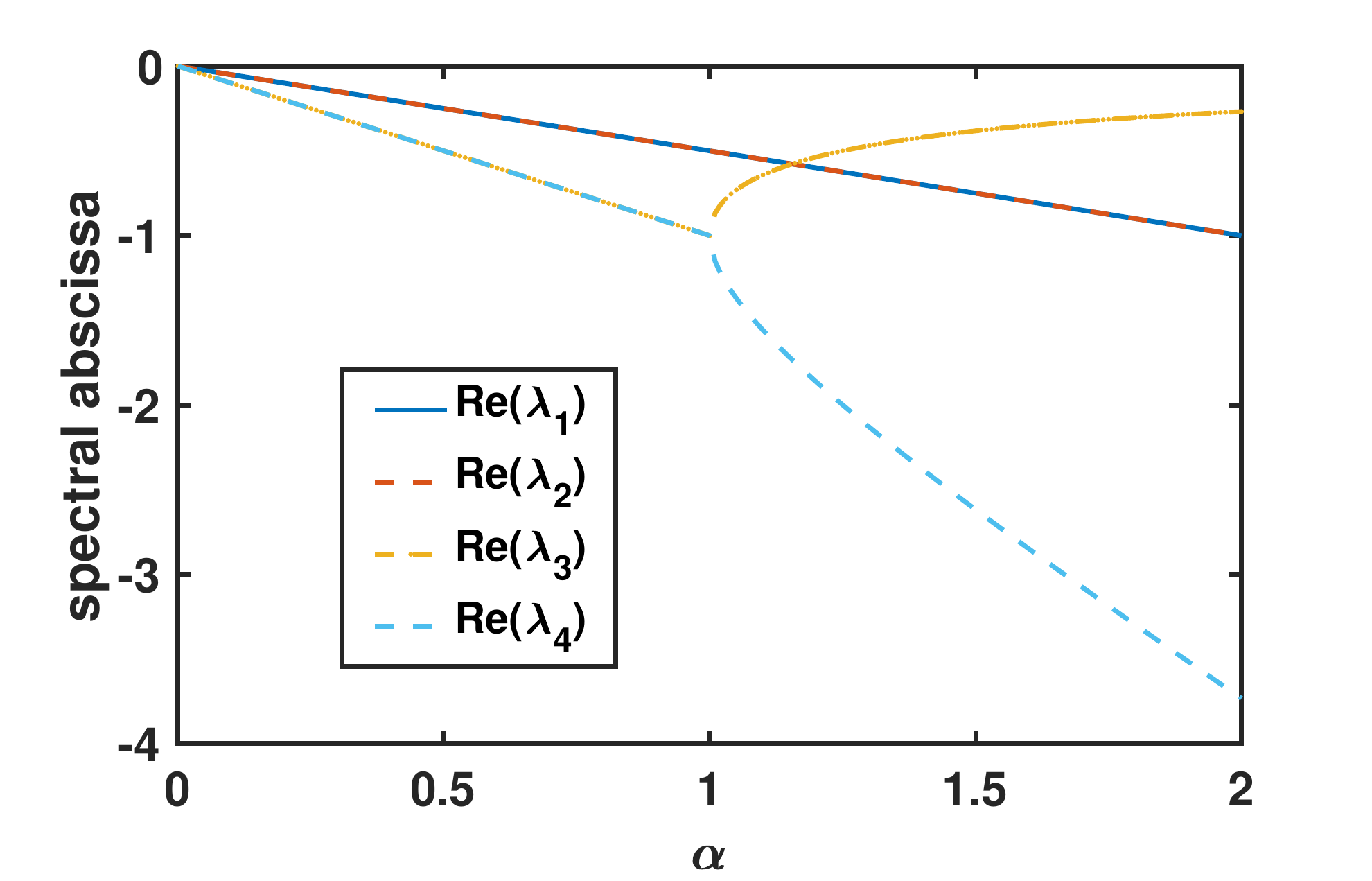}
	\caption{Real parts of eigenvalues defined by (\ref{linearSDE4x4}). The minimum spectral abscissa is attained at $\alpha\approx 1.1547$ when 3 of the eigenvalues with algebraic multiplicity one have the same real part.}\label{fig:4x4EWP}
\end{figure}
%

The following can be shown if the potential energy matrix $K$ and friction matrix $\gamma$ have a special form.
\begin{lem} \label{lem:optRatio}
    Let $K$ and $\gamma$ commute. Then there exists a matrix $S \in \mathbb{R}^{n \times n}$ such that $SKS^{-1} = diag(k_1,\ldots,k_n)$ and $S\gamma S^{-1} = diag(g_1,\ldots,g_n)$, where without loss of generality $g_1 \leq g_2 \leq \ldots \leq g_n$.  The optimal convergence
    rate is then attained for \[\frac{\beta}{\bar\beta}= \min\limits_{j \in \left\{1,\ldots, n\right\}}\sqrt{\frac{4 k_j}{g_j^2 - (g_j - g_1)^2}}.\] 
\end{lem}
\begin{proof}
Following the above considerations (see also \cite[Theorem 4.9]{ArnoldErb}) we are looking for the maximum real part of the eigenvalues of the drift matrix $A$ that will be minimised over $\alpha = \beta/\bar\beta$. To this end we  consider the eigenvalue problem for the matrix $A$
\[\begin{pmatrix} 0 & I_n \\ -K & - \alpha \gamma \end{pmatrix}  \begin{pmatrix} v_1 \\ v_2 \end{pmatrix} = \lambda \begin{pmatrix} v_1 \\ v_2 \end{pmatrix},\]
where $\alpha=\beta/\bar\beta$. Here $\lambda \neq 0 $ is an eigenvalue iff \[ v_1 = \lambda^{-1} v_2 \,, \quad -K v_2 = (\lambda^2 + \alpha \lambda \gamma) v_2\,.\] Rewriting the last equation in the basis given by the simultaneous diagonalization $S$ yields the following system of equations
\begin{align*}
- k_i = \lambda^2 + \lambda \alpha g_i\,, \qquad i=1,\ldots,n.
\end{align*}
Hence, for each $i$ we find a pair of eigenvalues   $\lambda_{i}^{1,2}(\alpha) = - \frac{\alpha g_i}{2} \pm \sqrt{ \frac{\alpha^2 g_i^2}{4} - k_i }\,.$ 
Since we are looking for the maximum real part, we only consider 
\begin{align*}
\lambda_i(\alpha) := \lambda_i^1(\alpha) = - \frac{\alpha g_i}{2} + \sqrt{ \frac{\alpha^2 g_i^2}{4} - k_i }\,.\end{align*}
For $\alpha \in (0,r)$, where $r = \min\limits_{i \in \left\{1,\ldots,n\right\}} \sqrt{\frac{4k_i}{g_i^2}}$ we have
\begin{align*}\min\limits_{i \in \left\{1,\ldots,n\right\}} {\rm Re}(\lambda_i(\alpha)) = {\rm Re}(\lambda_1(\alpha))= - \frac{\alpha g_1}{2} 
\end{align*}
which is decaying as a function of $\alpha$.
For $\alpha\geq r$ the discriminant of at least one eigenvalue $\lambda_j(\alpha)$ is positive and $ \lambda_j(\alpha)$ is thus monotonically increasing as a function of $\alpha$ leading to an intersection with the line $ - \frac{\alpha g_1}{2}$ at
 $\alpha_j  = \sqrt{\frac{4 k_j}{g_j^2 - (g_j - g_1)^2}}$. The minimum of the maximum real part of the eigenvalues of $A$ will thus be determined by the first intersection of the line $  - \frac{\alpha g_1}{2}$ with any other eigenvalue curve $\lambda_j(\alpha)$ , i.e.\ at
\begin{align*} \alpha^* = \min\limits_{j \in \left\{1,\ldots,n \right\}} \alpha_j = \min\limits_{j \in \left\{1,\ldots,n \right\}} \sqrt{\frac{4 k_j}{g_j^2 - (g_j - g_1)^2}}.
\end{align*}
Note that in case $r = \sqrt{4k_1/g_1^2}$ there is no intersection at $\alpha^ *=r = \argmin\limits_{\alpha} {\rm Re(\lambda_1(\alpha))}$, but ${\rm Re(\lambda_1(\alpha))}$ attains its global minimum at $\alpha^*$.
\end{proof}

\begin{rem}
The optimal temperature ratio (\ref{optRatio}) is a special case of Lemma \ref{lem:optRatio} for $n=1$. 
\end{rem}

\subsection{Fast sampling of a bistable system}

\begin{figure}
	\includegraphics[width=0.49\textwidth]{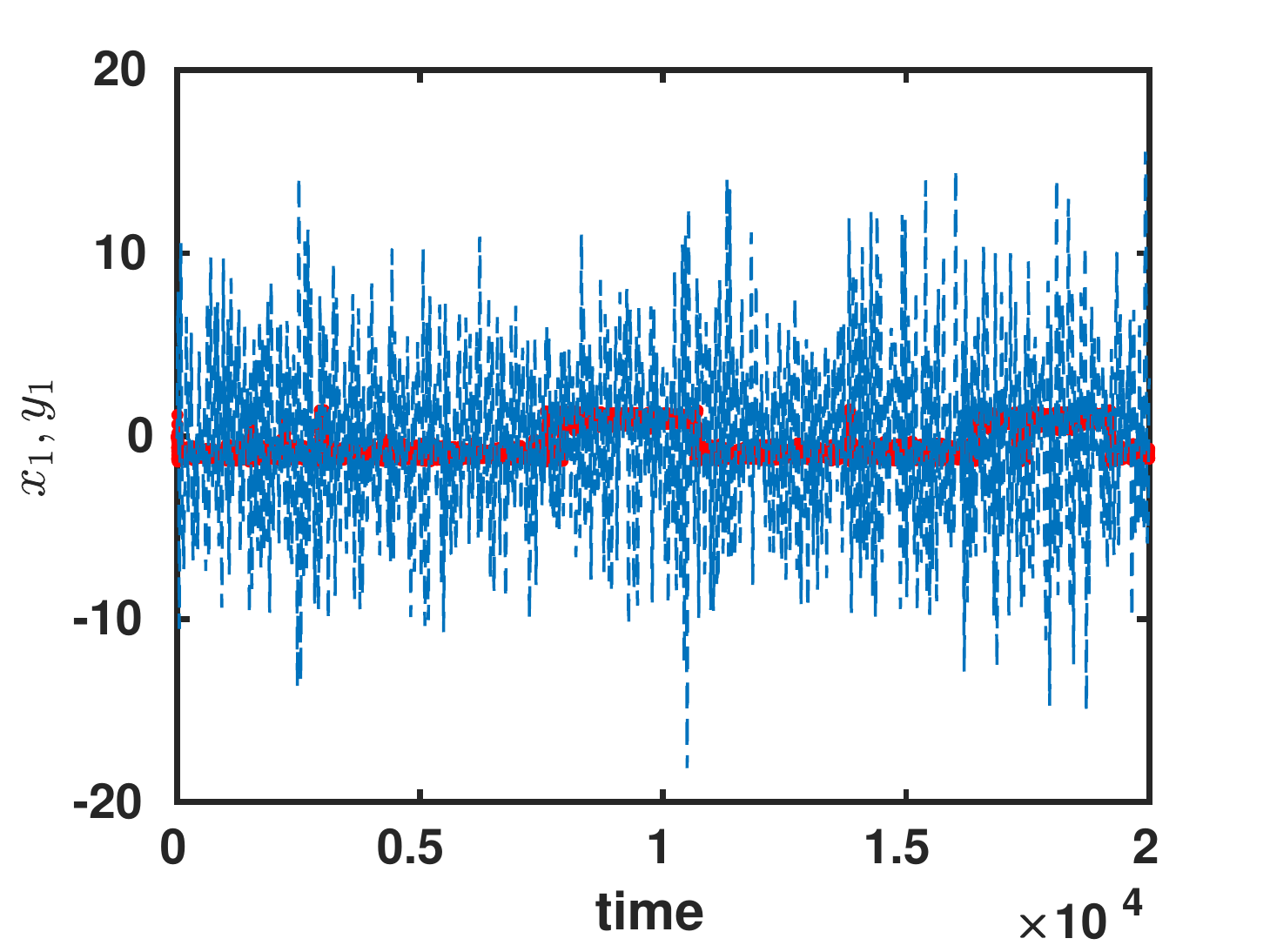}
	\includegraphics[width=0.49\textwidth]{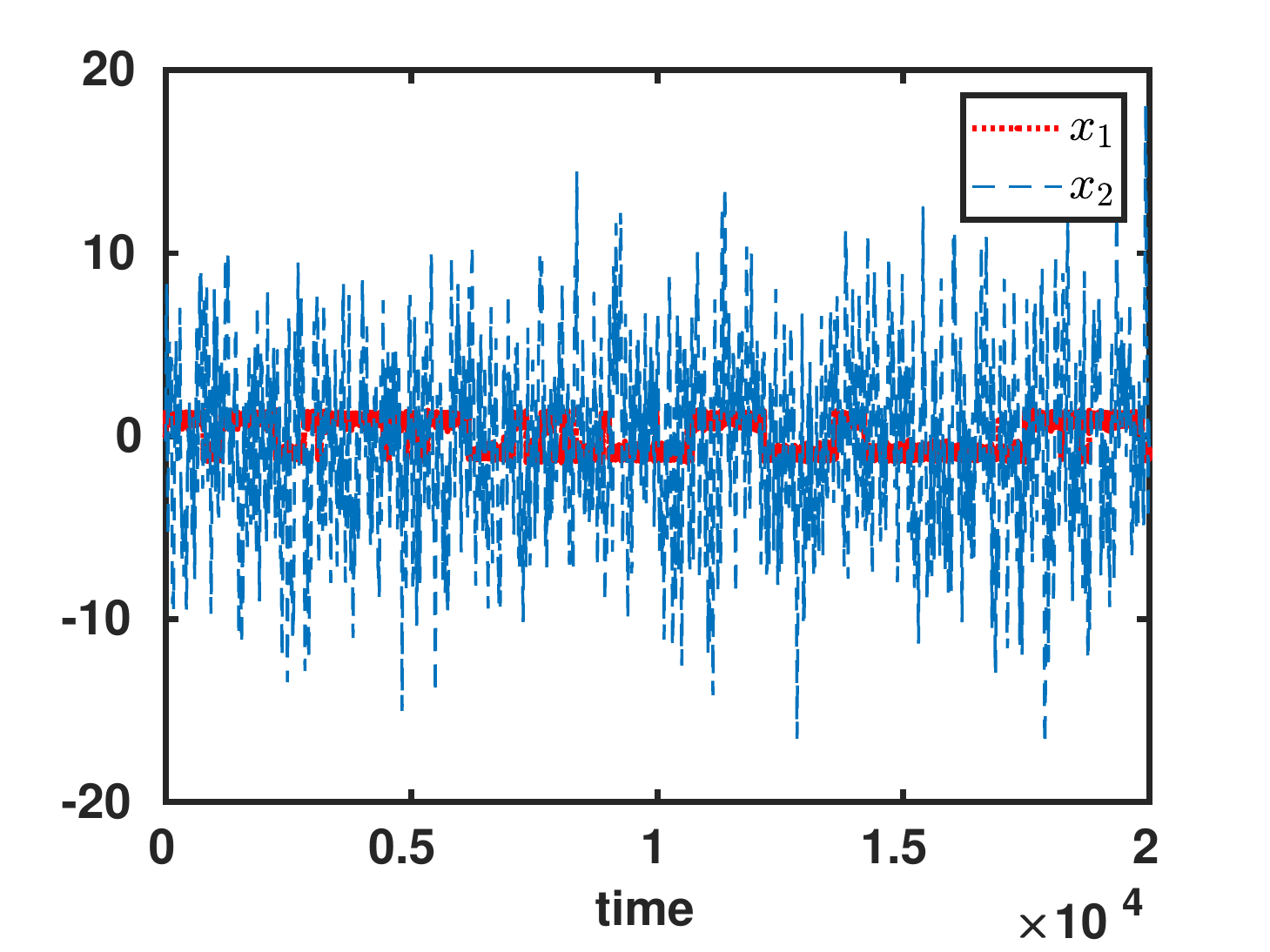}
	\caption{Two representative realisations of the controlled Langevin system at inverse temperature $\bar{\beta}=1.0$ (right panel) and the uncontrolled reference dynamics at target inverse temperature $\beta=5.0$ (left panel). The red curves show the metastable configuration $q$, the blue curves show the configuration  $\xi_1$.}\label{fig:trajectories}
\end{figure}

We now illustrate the previous considerations for linear systems with a nonlinear toy system at low friction. To this end, we consider the Langevin equation of dimension $n=d+1$ ($d=100$), with the potential 
\begin{align*}
V(x) = \frac{1}{2}(q^{2}-1)^{2} + \frac{k}{2}\sum_{i=1}^d (\xi_i-q)^{2}\,,\quad x=(q,\xi)\in\R\times\R^{d}\,,
\end{align*}
for a $k>0$. The friction coefficient $\gamma=\gamma^T>0$ is given by a tridiagonal matrix having 0.04 on the main diagonal and 0.02 above and below. In what follows we abbreviate the configuration components as $x=(q,\xi)$ and the velocities or momenta (strictly speaking) as $y=(u,\zeta)$. 

We run the controlled Langevin dynamics for $\bar{\beta}=1.0$ and compare with an uncontrolled simulation at target temperature $\beta=5.0$. Note that in this situation---other than in the linear case---no analytical result regarding the optimal temperature is available, so we run several numerical tests first, from which we determine a temperature ratio $\beta/\bar{\beta}\approx 5$ that shows good convergence towards the equilibrium distribution in the eyeball norm. 

For all numerical simulations, we use a variant of the BAOAB scheme with step size $\Delta t=5\cdot 10^{-3}$ (see \cite{baoab}), and we record the empirical $q$-marginals as well as the velocity autocorrelation function for the first coordinate after $4\,000\,000$ steps using a subsample of $20\,000$ points.  Figure \ref{fig:trajectories} shows a typical trajectory for the first and the third components $q$ and $\xi_{1}$. While $\xi_{1}$ is essentially Gaussian, we observe that the mixing rate between the metastable sets, i.e.\ the 2 wells of the double-well potential is much higher for the controlled simulation at $\bar{\beta}<\beta$ than for the uncontrolled simulation with $\bar{\beta}=\beta$; this is in accordance with the well converged empirical distribution shown in the left panel of Figure \ref{fig:comparison} that should be contrasted with the result of the uncontrolled simulation that has not yet converged. 

\paragraph*{Decay of correlations}

For further comparison, we also compute the velocity autocorrelation function for the  $q$-component, i.e.\ $U=dQ/dt$ where we use the shorthands $Q=X_1$ and $U=Y_1$. Since the Langevin equation is degenerate, with the noise driving only the $y$-components, the velocity autocorrelation exists. For all practical purposes the corresponding time series is weakly stationary with mean zero and variance $\vartheta^2>0$, and thus the velocity autocorrelation is unambigously defined by 
\begin{align*}
C_U(s) = \frac{\bE[U_tU_{t+s}]}{\vartheta^2} 
\end{align*}
where the expectation is understood over the (stationary) marginal distribution of the component $U_t$. Using ergodicity, we can approximate the autocorrelation function by 
\begin{align*}
C_U(s) \approx \frac{\displaystyle\int_0^T U_tU_{t+s}\,dt}{\displaystyle\int_0^T U_t^2\,dt}\,,
\end{align*}
for some sufficiently large $T\in(0,\infty)$, with the integrals being replaced by sums when discrete time series data are used (as is the case here).  

Interestingly enough, despite the much higher mixing rate in the controlled simulation due to the higher noise level, even the velocity autocorrelation functions of the two simulations for $\bar{\beta}<\beta$ and $\bar{\beta}=\beta$ are in good qualitative agreement (see right panel of Figure \ref{fig:comparison}). The fact that the autocorrelation function of the controlled simulation decays faster than for the uncontrolled simulation is in line with the observation in Section \ref{sec:limits} that the temperature ratio plays the role of a time scale separation parameter; cf.~equation (\ref{langevinSDEalt3}).

\begin{figure}
\includegraphics[width=0.49\textwidth]{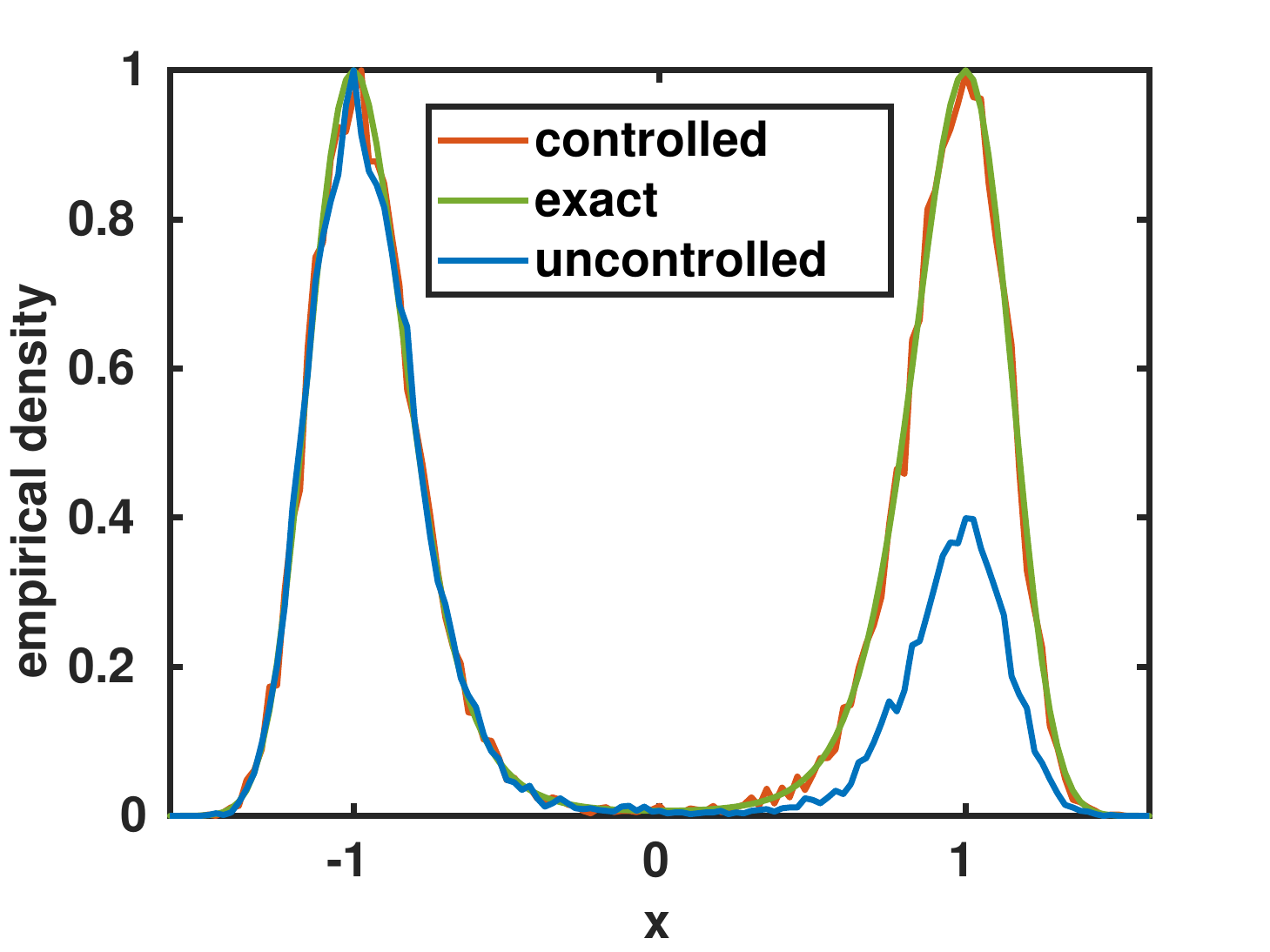}
\includegraphics[width=0.49\textwidth]{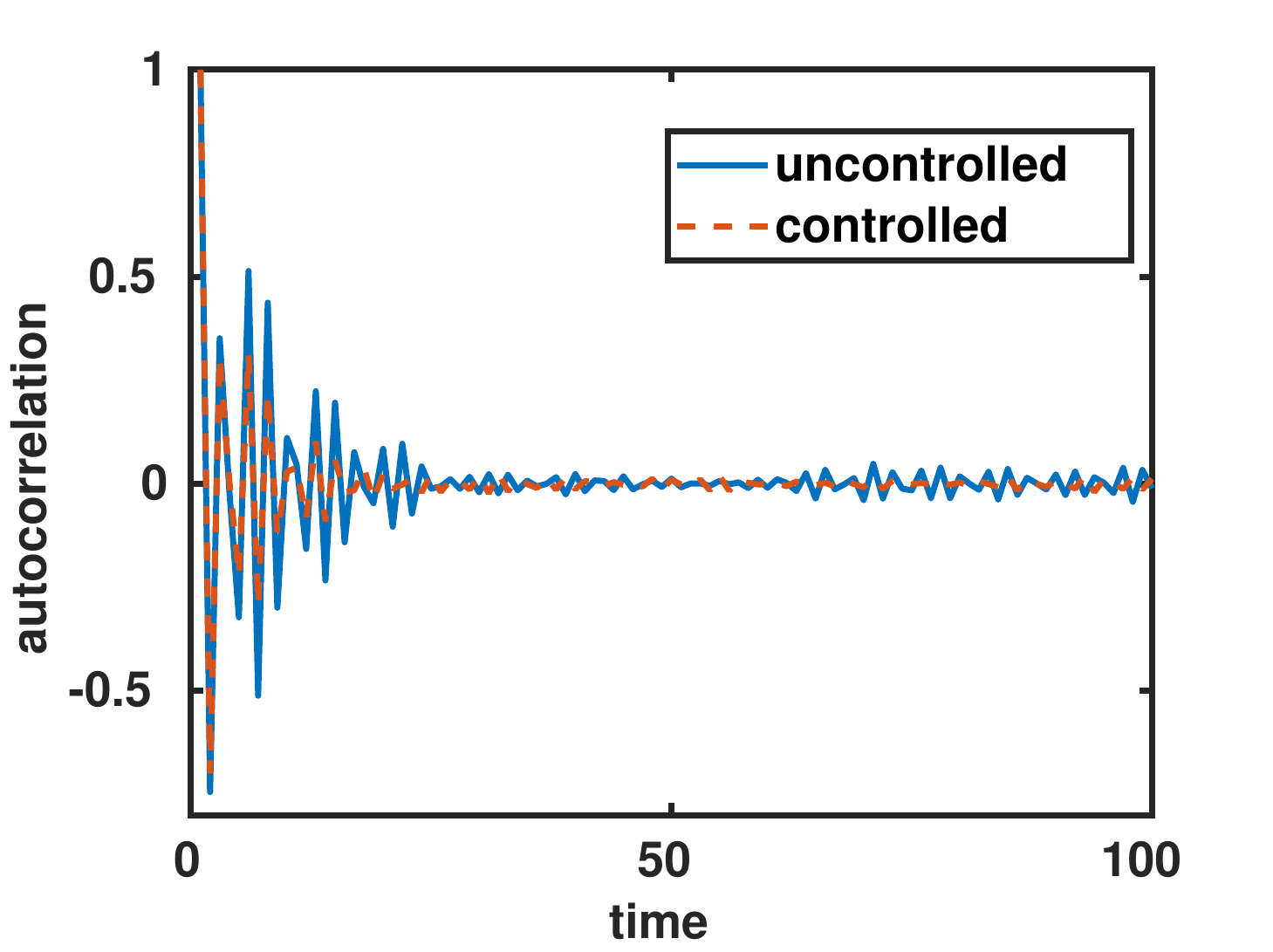}
\caption{Comparison of controlled and uncontrolled simulations: The left panel shows the resulting empirical $q$-marginals at $t=20\,000$, whereas the right panel shows the velocity autocorrelations of  $y_1=dq/dt$.}\label{fig:comparison}
\end{figure}

\subsection{Fast cooling and non-commutativity of limits}

Now we consider the interaction potential $V\colon\R^{n}\to \R$ of a Lennard-Jones cluster that consists of pairwise interaction terms  
\begin{equation*}
v_{ij} = 4\varepsilon\Bigl[\Bigl(\frac{\sigma}{r_{ij}}\Bigr)^{12} - \Bigl(\frac{\sigma}{r_{ij}}\Bigr)^{6}\Bigr]\,,
\end{equation*}
where $\varepsilon,\sigma>0$ are parameters and $r_{ij}=|x^{(i)}-x^{(j)}|$, $i\neq j$ denotes the distance between the $i$-th and the $j$-th particle with coordinates $x^{(i)},x^{(j)}\in\R^{d}$, respectively. The function $V$ is a sum of all pairwise interaction terms. 
 
Suppose we want to find the unique minimum energy configuration $x^{*}=(x^{(1)},\ldots,x^{(N)})\in \R^{n}$ by gradient descent in $V$ where $n=dN$ with $N$ denoting the number of particles and $d$ being the spatial dimension---typically $1,2$ or $3$. (Note that $x^{*}$ is unique up to translations and rotations, i.e.\ up to rigid body symmetries.)

Since the potential has many irrelevant local minima a gradient descent algorithm will get stuck in the local minimum that is closest to its starting value. The Langevin dynamics (\ref{langevinSDEalt2}) approximates a gradient descent as $\eps\to 0$ that gets stuck in local minima, but for every finite $\eps>0$, the dynamics has has a unique ergodic invariant measure $\mu_{\infty}^{\eps}$ with density
\begin{equation}\label{invMeasure2}
\rho^{\eps}_{\infty} = \frac{1}{Z^{\eps}}\exp\Bigl(-\frac{1}{\eps}\bar{H}\Bigr)\,,\quad Z^{\eps} = \int \exp\Bigl(-\frac{1}{\eps}\bar{H}(x,y)\Bigr)dxdy\,,
\end{equation}
where $ \bar{H}=\bar{\beta}H$ is independent of $\eps$. (The representation (\ref{invMeasure2}) readily follows from Lemma \ref{lem:invMeasure2} by rescaling the Hamiltonian.) By Laplace's principle, we have  (e.g., see \cite[Theorem 4.3.1]{DemboZeitouni})
\begin{equation}\label{laplace}
\lim_{\eps\to 0}\int \rho^{\eps}_{\infty}(x,y) f(x) \, dxdy  = f(x^{*})\quad f\in C_{b}(\R^{d}) \,.
\end{equation}
In other words the invariant measure $\mu_{\infty}^{\eps}$ concentrates at the unique global minimum $x^{*}=\argmin_{x}V(x)$ as $\eps\to 0$. Since, by Theorem \ref{thm:SGD}, the limit $\eps$ of the Langevin equation on any finite time interval $[0,T]$ converges to the deterministic gradient system that has the property to converge to the nearest local minimum as $t\to \infty$, it is a straight consequence that the limits $t\to\infty$ and $\eps\to 0$ in the two-temperature Langevin dynamics do not commute. 

We will exploit  this property to simulate the two-temperature Langevin dynamics (\ref{langevinSDEalt2}) at large, but finite temperature separation limit, so as to mimic a stochastic gradient descent that is likely to converge to a minimum that is close to the global one in the sense that $V(X_{T})\approx V(x^{*})$ whenever $T$ is sufficiently large.

\paragraph*{Cooling a Lennard-Jones system}

We seek the unique minimum energy configuration of a Lennard-Jones cluster of $N=7$ particles---or a close approximation thereof.  The idea now is to exploit the fact that, for finite $\eps$, the Langevin dynamics  (\ref{langevinSDEalt2}) is close to a stochastic gradient descent with momentum (like ADAM or its variants) that has the property of being absorbed by a very low minimum on $\cO(1)$ time scales and to remain there over $\cO(\exp(1/\eps))$ time scales. We expect that for sufficiently small $\eps>0$, the dynamics (\ref{langevinSDEalt2}) will essentially fluctuate around the global minimum $x^{*}$. 

As before we employ the BAOAB scheme with step size $\Delta t=5\cdot 10^{-3}$ and a total number of $2\,000\,000$ time steps, where we  record only every 200th step. The simulation temperature is fixed at $\bar{\beta}=1.0$, whereas the target temperature is varied between $\beta=\bar{\beta}$ and $\beta=\bar{\beta}/100$. Figure \ref{fig:LJ7} shows the results of two simulations for $\beta=1.0$ (left panel) and $\beta=0.01$ (right panel), both starting from the same high energy initial configuration. It can be seen that the simulation in the stochastic gradient descent regime for $\beta=0.01$ quickly goes down the energy ladder by roughly 11 to 12 orders of magnitude, similar to the uncontrolled equilibrium simulation, but then stabilises in a low-lying local minimum (see energy plot in the right panel of Figure \ref{fig:LJ7}). Since the target temperature is small compared to the typical energy barrier, the probability to escape from the attractor of the local minimum during simulation time is negligible. The uncontrolled simulation even reaches a yet lower energy state, but it oscillates wildly between local energy minima that span three orders of magnitude in energy, with a high escape probability (see energy plot in the left panel).

\begin{figure}
	\includegraphics[width=0.49\textwidth]{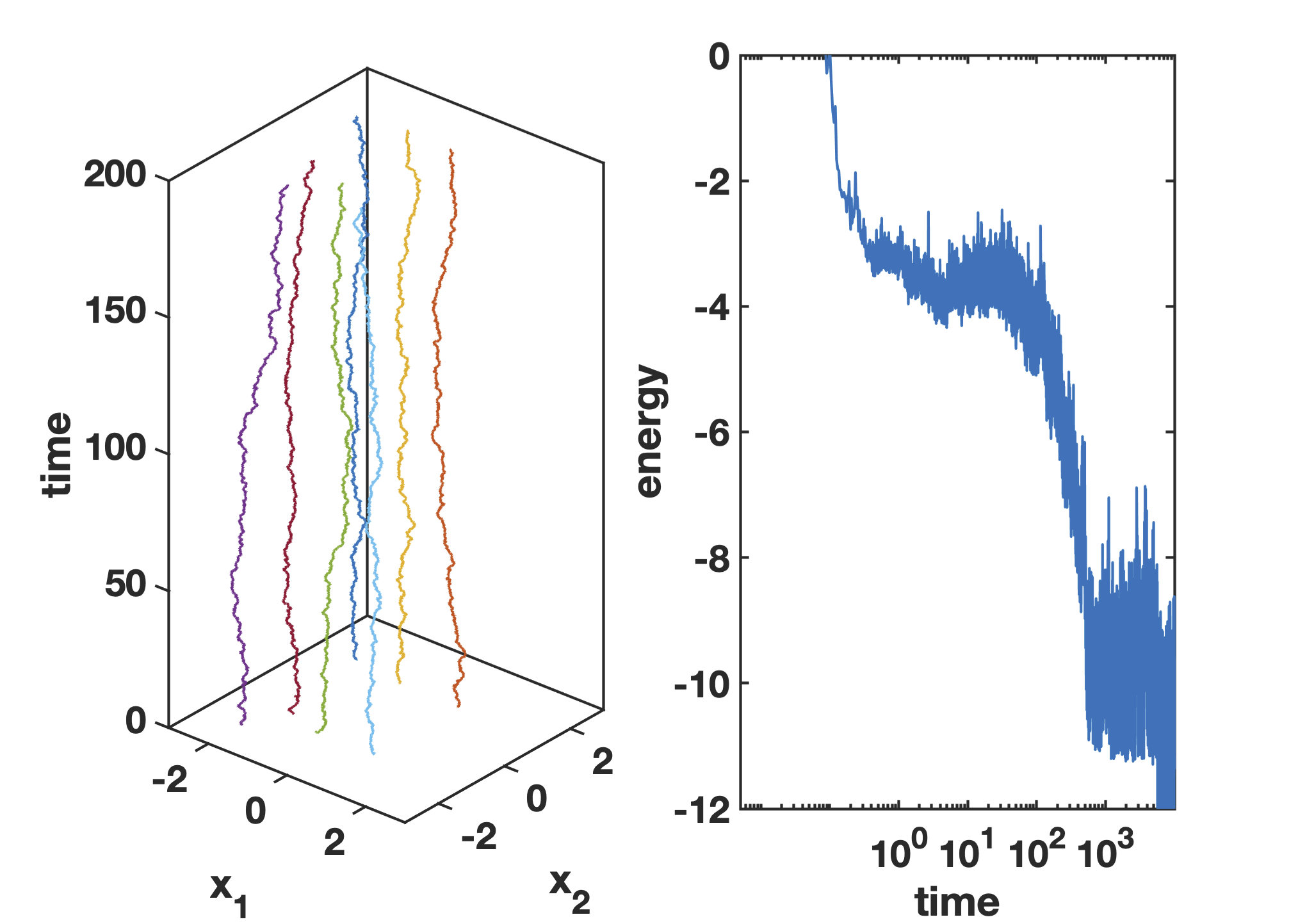}
	\includegraphics[width=0.49\textwidth]{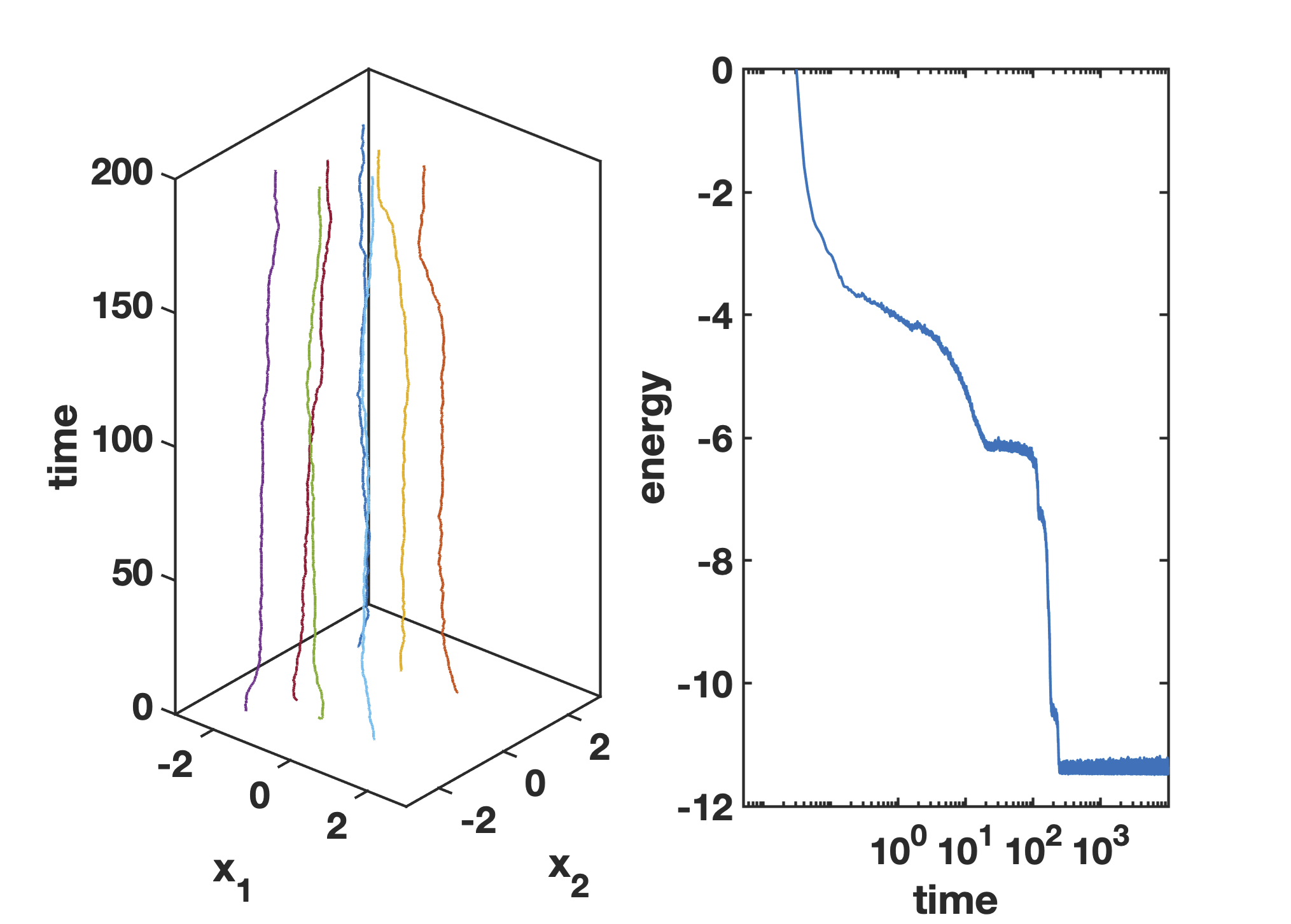}
\caption{Langevin simulation of a Lennard-Jones cluster at fixed simulation temperature $\bar{\beta}=1$: high friction regime (left panel) vs.~stochastic gradient descent regime at small target temperature (right panel)}\label{fig:LJ7}
\end{figure}

\section{Discussion}\label{sec:discussion}

We have studied the asymptotic properties of a feedback-controlled underdamped Langevin equation at high (simulation) temperature that is controlled by a linear feedback force. The feedback force acts as a friction and balances the noise so that the dynamics weakly converges to a thermal equilibrium at a lower (target) temperature. The rationale of choosing a simulation temperature that is higher than the target temperature is that the dynamics is less likely to get trapped in local minima of the Hamiltonian compared to simulating the system at the target temperature. The approach is reminiscent of the Schrödinger bridge problem on an infinite time horizon for diffusions with different noise coefficients. By using a functional inequality approach (FIR inequality) we have identified a control that enjoys a minimum dissipation property and that has been derived in \cite{Pavon2015} using different arguments. 

The fact that the simulation temperature can be different from the target temperature opens up various possibilities for designing sampling or global optimisation algorithms. For example, it is possible to ask what is the optimal temperature ratio that maximises the speed of the convergence to the desired target distribution. We have given only a partial answer to this question for Gaussian processes and collected further empirical evidence based on numerical simulations; cf. also \cite{Eberle2019} for a theoretical analysis. It should be mentioned that by choosing a too high simulation temperature, the speed of convergence can even be reduced as compared to the case when the simulation and the target temperatures are equal. On the other hand, if one is interested in exploring the minima of the underlying Hamiltonian, one could lower the target temperature, while holding the simulation temperature fixed, so that the invariant probability measure concentrates in the low-lying minima. We have numerically demonstrated with a benchmark example from molecular dynamics that the controlled Langevin dynamics can be used to efficiently solve high-dimensional and non-convex minimisation problems.  

From an algorithmic point of view, a relevant question then is whether it is possible to reach a (theoretically) arbitrary speed of convergence for sampling or optimisation by sending the temperature ratio to infinity. The answer is clearly negative as has been already pointed out in \cite{Duncan2017,Eberle2019}. More specifically, it turns out that the controlled dynamics converges either to the overdamped Langevin equation or to a deterministic gradient flow, depending on whether the target or the simulation temperature is held fixed. As a consequence, the ergodic limit and the large temperature separation limit do not commute in general, and the feedback-controlled Langevin dynamics interpolates between sampling and stochastic gradient descent. 

Future work shall address the question of optimising temperature ratios beyond the Gaussian setting or choosing more general feedback laws. The latter may be done by rephrasing the problem as an optimal control problem for the Fokker-Planck equation which relates this line of research to recent work on mean-field approaches for optimisation (e.g. \cite{Pinneau2017}).

\section*{Acknowledgement}

This work has been partially supported by the Collaborative Research Center \emph{Scaling Cascades in Complex Systems} (DFG-SFB 1114) through project A05 and by the MATH+ Cluster of Excellence (DFG-EXC 2046) through the projects EP4-4 and EF4-6. 

\appendix

\section{FIR inequality and minimum dissipation functional}\label{sec:FIR}

We want to study the density evolution of the controlled Langevin dynamics (\ref{langevinSDE}) and compare it to an uncontrolled simulation at the target temperature $1/\beta$: 
\begin{align}\label{langevinSDE2}
\begin{aligned}
dQ_{t} & = P_{t}\,dt\,,\quad Q_{0}=x\\
dP_{t} & = -\left(\nabla V(Q_{t}) + \gamma P_{t} \right)dt + \bar{\sigma} dW_{t}\,,\quad P_{0}=y\,,
\end{aligned}
\end{align}
with
\begin{align*}
\bar{\sigma}=\sqrt{\frac{\bar\beta}{\beta}}\,\sigma\,.
\end{align*}

\subsection{FIR inequality}

Obviously, the path measures of (\ref{langevinSDE}) and (\ref{langevinSDE2}) are mutually singular because the equations are not driven by the same Brownian motion, and therefore we cannot apply Girsanov's theorem to do a pathwise comparison. (The noise coefficients are different.)  Nevertheless it makes sense to compare the time $t$ marginals of the two equations. 
To this end, we write the Fokker-Planck equation associated with (\ref{langevinSDE}) as a continuity equation
\begin{equation}\label{langevinFPE}
\frac{\partial \rho}{\partial t} + \nabla\cdot\left(v\rho\right) = 0\,,
\end{equation}
with the time-dependent vector field
\begin{align*}
v_t = (J - \Gamma)\nabla H + U_t - \frac{1}{2}\Sigma\Sigma^\top \nabla\log\rho_t\,.
\end{align*}
We use the shorthand notation $\rho_t=\rho(\cdot,t)$ for the smooth probability density of $(X_t,Y_t)\in\R^n\times\R^n$ at time $t>0$ and  $U=(0,\sigma u)^\top\in\R^{n}\times \R^n$ to denote the control that is acting on the momentum equation only, and we have introduced the ($2n\times 2n$)-matrices
\begin{align*}
J = \begin{pmatrix}
0_{n} & I_{n}\\ -I_{n} & 0_{n}
\end{pmatrix}
,\quad 
\Gamma = \begin{pmatrix}
0_{n} & 0_{n}\\ 0_{n} & \gamma
\end{pmatrix},\quad 
\Sigma = \begin{pmatrix}
0_{n}\\ \sigma  
\end{pmatrix}.
\end{align*}
We further let 
\begin{equation}\label{langevinFPE2}
\frac{\partial \eta}{\partial t} + \nabla\cdot\left(\bar{v}\eta\right) = 0\,,
\end{equation}
with 
\begin{align*}
\bar{v}_t = (J - \Gamma)\nabla H - \frac{1}{2}\bar{\Sigma}\bar{\Sigma}^\top \nabla\log\eta_t\,.
\end{align*}
denote the Fokker-Planck equation associated with the reference Langevin dynamics (\ref{langevinSDE2}) at the target (inverse) temperature $\beta$ with $\bar\Sigma = \begin{pmatrix}
0_{n}\\ \bar\sigma  
\end{pmatrix}.
$ 

\begin{defn}[Relative entropy]\label{defn:kl}
	We define the relative entropy between $\rho_t$ and $\eta_t$ as
	\begin{equation}\label{KL}
	\KL(\eta_t|\rho_t) = \int_{\R^{2n}}\log\frac{\eta_{t}(z)}{\rho_{t}(z)}\,\eta_{t}(z)\,dz,
	\end{equation}
	provided that
	\begin{align*}
	\int_{\{\rho_t(z)=0\}} \eta_t(z)\,dz = 0\,,
	\end{align*}
	and otherwise we set $\KL(\eta_t|\rho_t)=\infty$.
\end{defn}
 
Assuming that $\eta_t$ is sufficiently decaying at infinity\footnote{Specifically, we assume that \begin{align*}
	\lim_{|z|\to\infty} v_t(z)\eta_t(z) = 0\,,\quad 	\lim_{|z|\to\infty} \bar{v}_t(z)\eta_t(z) = 0\,,\quad
	\lim_{|z|\to\infty}  \bar{v}_t(z)\eta_t(z)\log\left(\frac{\eta_t(z)}{\rho_t(z)}\right) = 0\,.
	\end{align*}} and that $\KL(\eta_t|\rho_t)<\infty$ for all $t>0$, it is straightforward to show that 
\begin{equation}\label{KLrate}
\frac{d}{dt} \KL(\eta_t|\rho_t) = \int_{\R^{2n}} \nabla\log\left(\frac{\eta_t}{\rho_t}\right)\cdot\left(\bar{v}_t-v_{t}\right)\eta_t\,dz\,,
\end{equation}
since (using a dot to denote the partial derivative \wrt $t$)
\begin{align*}
\frac{d}{dt} \KL(\eta_t|\rho_t) & = \int_{\R^{2n}} \left(\frac{\dot\eta_t}{\eta_t}-\frac{\dot\rho_t}{\rho_t}\right)\eta_t\,dz + \int_{\R^{2n}} \dot{\eta}_t\log\left(\frac{\eta_t}{\rho_t}\right)dz\\
 & = \int_{\R^{2n}} \left(-\nabla\cdot(\bar{v}\eta_t) + \frac{\eta_t}{\rho_t}\nabla\cdot(v\rho_t) - \log\left(\frac{\eta_t}{\rho_t}\right)\nabla\cdot(\bar{v}\eta_t)\right)dz\\
  & = \int_{\R^{2n}} \left(\nabla\log\left(\frac{\eta_t}{\rho_t}\right)\cdot(\bar{v}\eta_{t})-\nabla\left(\frac{\eta_t}{\rho_t}\right)\cdot(v\rho_t)\right)dz\\
  & = \int_{\R^{2n}} \nabla\log\left(\frac{\eta_t}{\rho_t}\right)\cdot\left(\bar{v}-v\right)\eta_t\,dz\,.
\end{align*}
Here we have substituted (\ref{langevinFPE}) and (\ref{langevinFPE2}) in the second equality and used integration by parts in the third equality, assuming fast decay of the density $\eta_t$ which allowed us to ignore boundary terms.  

\begin{thm}[FIR inequality, cf.~\cite{Sharma2016}]\label{thm:FIR}
	It holds, for all $\tau>0$, 
	\begin{equation}\label{FIR}
	\frac{d}{dt} \KL(\eta_t|\rho_t)\le  \frac{\tau - 1}{2}I(\eta_{t}|\rho_{t}) + \frac{1}{2\tau} S(\eta_{t})\,,
	\end{equation}
where 
\begin{equation}\label{FI}
I(\eta_t|\rho_t) = \int_{\R^{2n}} \Bigl|\nabla_{y}\log\Bigl(\frac{\eta_t}{\rho_t}\Bigr)\Bigr|_{\sigma\sigma^\top}^2\eta_t\,dz
\end{equation}
is the relative Fisher information, with $|z|_G=z^\top G z$ denoting the $G$-weighted Euclidean vector norm with a semidefinite weight matrix $G=\sigma\sigma^\top> 0$, and 
\begin{equation}\label{rate}
	S(\eta_t) = \int_{\R^{2n}}\Bigl|\sigma u  -  \frac{1}{2}\bigl(\sigma\sigma^\top-\bar{\sigma}\bar{\sigma}^\top \bigr)\nabla_{y}\log\eta_t\Bigr|_{(\sigma\sigma^{\top})^{-1}}^{2}\eta_t \,dz
\end{equation}
is a variant of the Donsker-Varadhan large deviations rate functional for the empirical measure generated by an SDE (e.g. see \cite[Theorem 2.5]{Duong2013}).
\end{thm}

\begin{proof}
We start from the representation (\ref{KLrate}). Substituting $v$ and $\bar{v}$ we obtain by adding a zero: 	
\begin{align*}
\frac{d}{dt} \KL(\eta_t|\rho_t) = & \int_{\R^{2n}} \nabla\log\Bigl(\frac{\eta_t}{\rho_t}\Bigr)\cdot\Bigl(-U -  \frac{1}{2}\bar{\Sigma}\bar{\Sigma}^\top \nabla\log\eta_t+ \frac{1}{2}\Sigma\Sigma^\top \nabla\log\rho_t \Bigr)\eta_t \,dz\\
 = &  - \frac{1}{2}\int_{\R^{2n}} \Bigl|\nabla\log\left(\frac{\eta_t}{\rho_t}\right)\Bigr|_{\Sigma\Sigma^\top}^2\eta_t\,dz\\ & +  \int_{\R^{2n}} \nabla\log\left(\frac{\eta_t}{\rho_t}\right)\cdot\Bigl(-U 
 +  \frac{1}{2}\bigl(\Sigma\Sigma^\top - \bar{\Sigma}\bar{\Sigma}^\top\bigr)\nabla\log\eta_t\Bigr)\eta_t \,dz\\
 = & -\frac{1}{2} I(\eta_t|\rho_t) +  \int_{\R^{2n}} \nabla_{y}\log\Bigl(\frac{\eta_t}{\rho_t}\Bigr)\cdot\Bigl(-\sigma u  +  \frac{1}{2}\bigl( \sigma\sigma^\top - \bar{\sigma}\bar{\sigma}^\top\bigr)\nabla_{y}\log\eta_t\Bigr)\eta_t \,dz
\end{align*}
where in the last line we have used that only the lower right block of the matrices $\Sigma,\bar{\Sigma}$ is nonzero. 
We set $\delta=\sigma\sigma^\top - \bar{\sigma}\bar{\sigma}^\top$ and consider the second integral. Using Cauchy-Schwarz and Young's inequality, we find \begin{align*}
I_{1}= & \int_{\R^{2n}} \nabla_{y}\log\Bigl(\frac{\eta_t}{\rho_t}\Bigr)\cdot\Bigl(-\sigma u  +  \frac{1}{2}\delta\nabla_{y}\log\eta_t\Bigr)\eta_t \,dz\\ 
 = & \int_{\R^{2n}} \nabla_y\Bigl(\frac{\eta_t}{\rho_t}\Bigr)\cdot\Bigl(-\sigma u  +  \frac{1}{2}\delta\nabla_y\log\eta_t\Bigr)\rho_t \,dz\\
= & \int_{\R^{2n}} \sigma\nabla_y\Bigl(\frac{\eta_t}{\rho_t}\Bigr)\frac{\rho_{t}}{\sqrt{\eta_{t}}}\cdot\sigma^{-1}\Bigl(-\sigma u  +  \frac{1}{2}\delta \nabla_y\log\eta_t\Bigr)\sqrt{\eta_{t}} \,dz\\
\le & \sqrt{\int_{\R^{2n}} \Bigl| \nabla_y\Bigl(\frac{\eta_t}{\rho_t}\Bigr)\Bigr|^{2}_{\sigma\sigma^{\top}}\frac{\rho_{t}^{2}}{\eta_{t}}\,dz}\sqrt{\int_{\R^{2n}}\Bigl|\sigma u  -  \frac{1}{2}\delta\nabla_y\log\eta_t\Bigr|_{(\sigma\sigma^{\top})^{-1}}^{2}\eta_t \,dz}\\
\le & \frac{\tau}{2}\int_{\R^{2n}} \Bigl| \nabla_y\Bigl(\frac{\eta_t}{\rho_t}\Bigr)\Bigr|^{2}_{\sigma\sigma^{\top}}\frac{\rho_{t}^{2}}{\eta_{t}}\,dz + \frac{1}{2\tau}\int_{\R^{2n}}\Bigl|\sigma u  -  \frac{1}{2}\delta\nabla_y\log\eta_t\Bigr|_{(\sigma\sigma^{\top})^{-1}}^{2}\eta_t \,dz
\end{align*}
for all $\tau>0$. Taking a closer look at the leftmost integral, we see that  
\begin{equation*}
I_{2} = \int_{\R^{2n}} \Bigl| \nabla_y\Bigl(\frac{\eta_t}{\rho_t}\Bigr)\Bigr|^{2}_{\sigma\sigma^{\top}}\frac{\rho_{t}^{2}}{\eta_{t}}\,dz 
= \int_{\R^{2n}} \Bigl| \nabla_y\log\Bigl(\frac{\eta_t}{\rho_t}\Bigr)\Bigr|^{2}_{\sigma\sigma^{\top}}\eta_{t}\,dz 
= I(\eta_{t}|\rho_{t})\,.
\end{equation*}
As a consequence, 
\begin{equation*}
\frac{d}{dt} \KL(\eta_t|\rho_t) \le \frac{\tau - 1}{2}I(\eta_{t}|\rho_{t}) + \frac{1}{2\tau} S(\eta_{t})
\end{equation*}

\end{proof}

\subsection{Entropy production rate}\label{sec:AEP}

In (\ref{FIR}) we may tighten the upper bound by minimising over $\tau>0$. From a practical point of view, however, it is more convenient to set $\tau=1$, as a consequence of which the Fisher information term vanishes.  

We now want to compute the entropy production of the controlled system. To this end, we set $\eta_{0}=\rho_{\infty}$ which implies that $\eta_{t}=\rho_{\infty}$ for all $t>0$. Using that $\delta=(1-\bar{\beta}/\beta)\sigma\sigma^{\top}$ and integrating (\ref{FIR}) we obtain the inequality  
\begin{equation}\label{integratedFIR}
\begin{aligned}
\KL(\rho_\infty|\rho_{T}) - \KL(\rho_{\infty}|\rho_{0}) & \le \frac{1}{2}\int_{0}^{T} \bE\!\left[\Bigl|\sigma u  -  \frac{1}{2}\delta\nabla_y\log\rho_\infty\Bigr|_{(\sigma\sigma^{\top})^{-1}}^{2}\right] dt\,,\\
& = \frac{1}{2}\int_{0}^{T} \bE\!\left[\Bigl|\sigma u  -  \frac{1}{2}(\bar{\beta}-\beta)\sigma\sigma^{\top} y\Bigr|_{(\sigma\sigma^{\top})^{-1}}^{2}\right] dt\,,
\end{aligned}
\end{equation}
where the expectation is taken over the invariant density $\rho_{\infty}\propto\exp(-\beta H)$. The term inside the integral is the (asymptotic) entropy production rate, and we define 
\begin{align}\label{AEP}
R(u)  = \frac{1}{2}\,\bE\!\left[\Bigl|\sigma u  -  \frac{1}{2}(\bar{\beta}-\beta)\sigma\sigma^{\top} y\Bigr|_{(\sigma\sigma^{\top})^{-1}}^{2}\right]\,.
\end{align}
Note that when $u$ is admissible, then $R(u)$ shows up as asymptotic entropy production rate as $T\to\infty$, even without the assumption $\eta_{0}=\rho_{\infty}$, because the law $\eta_{t}$ of the uncontrolled dynamics (\ref{langevinSDE2}) asymptotically approaches the target density $\rho_{\infty}$ as $t\to\infty$, independently of the initial distribution. Hence the name \emph{asymptotic entropy production rate}.

\section{Proof of Theorem~\ref{thm:SGD}}\label{app:SGD}

The proof below is an adaption of~\cite[Prop. 2.14]{mathias2010free} for a different scaling.
Since~\eqref{langevinSDEalt2Y} is a non-homogenous linear equation in $Y$, we have the explicit solution
\begin{equation*}
Y_t = e^{-\frac{t}{\eps^2}\gamma  }y - \frac{1}{\eps}\int_0^t e^{-\frac{t-s}{\eps^2} \gamma } \nabla V(X_s)ds +\frac{1}{\sqrt{\eps}}\int_0^te^{-\frac{t-s}{\eps^2} \gamma }\sigma  dW_s .
\end{equation*}
Integration of \eqref{langevinSDEalt2X} thus yields
\begin{align*}
&X_t = x + \frac{1}{\eps} \int_0^t Y_s ds \\
&= x + \frac{1}{\eps}\int_0^t  e^{-\frac{s}{\eps^2}\gamma  }y ds - \frac{1}{\eps^2}\int_0^t \int_0^s e^{-\frac{s-r}{\eps^2} \gamma } \nabla V(X_r)dr ds +\frac{1}{\eps^{3/2}}\int_0^t \int_0^s e^{-\frac{s-r}{\eps^2} \gamma } \sigma  dW_r ds \\
&= x + \eps \gamma^{-1} (I_n - e^{-\frac{t}{\eps^2}\gamma}) y
- \frac{1}{\eps^2}\int_0^t \int_r^t e^{-\frac{s-r}{\eps^2} \gamma }ds \nabla V(X_r) dr +\frac{1}{\eps^{3/2}}\int_0^t \int_r^t e^{-\frac{s-r}{\eps^2} \gamma }  \sigma ds dW_r  \\
&=  x + \eps \gamma^{-1} (I_n - e^{-\frac{t}{\eps^2}\gamma})  y
- \int_0^t \gamma^{-1}(I_n - e^{-\frac{t-r}{\eps^2} \gamma }) \nabla V(X_r) dr 
+\sqrt{\eps}\int_0^t   \gamma^{-1}(I_n - e^{-\frac{t-r}{\eps^2} \gamma })\sigma dW_r. 
\end{align*}
Together with ~\eqref{SGD} this gives 
\begin{align*}
    X_t - x_t &= \eps\gamma^{-1} (I_n - e^{- \frac{t}{\eps^2} \gamma})y 
    - \int_0^t \gamma^{-1}(I_n - e^{-\frac{t-r}{\eps^2} \gamma }) (\nabla V(X_r) - \nabla V(x_r))dr \\
    &\quad +\int_0^t \gamma^{-1} e^{-\frac{t-r}{\eps^2}\gamma} \nabla V(x_r) dr 
    +\sqrt{\eps}\int_0^t \gamma^{-1}(I_n - e^{-\frac{t-r}{\eps^2}\gamma})\sigma dW_r .
\end{align*}
Let us consider these terms separately and let $t \in [0,T]$. Also note that $|I_n - e^{-\frac{t}{\eps^2}\gamma}|_F = \sqrt{\sum_{i \leq n} (1-e^{-\frac{t}{\eps^2}\lambda_i})^2} \leq \sqrt{n}$ for $t> 0$, where $\lambda_i,\ i=1,\ldots,n$ are the eigenvalues of $\gamma = \gamma^T >0$, and $|A|_F^2 = tr(A^TA)$. For the first term it holds that 
\begin{equation*}
    |\eps\gamma^{-1} (I_n - e^{- \frac{t}{\eps^2} \gamma})y |
    \leq \eps \sqrt{n} |\gamma^{-1}|_F |y| .
\end{equation*}
Using the Lipschitz continuity of $\nabla V$ the second term can be bounded by
\begin{align*}
    \Bigl|\int_0^t \gamma^{-1}(I_n - e^{-\frac{t-r}{\eps^2} \gamma }) (\nabla V(X_r) - \nabla V(x_r))dr \Bigr|
    &\leq \sqrt{n} |\gamma^{-1}|_F L_V \int_0^t |X_r -x_r| dr .
\end{align*}
For the third term we use, denoting the eigenvalues of $\gamma$ by $\lambda_i$, that \[|e^{-\frac{t-r}{\eps^2}\gamma}|_F = \Bigl( \sum_{i \leq n} e^{-2 \frac{t-r}{\eps^2} \lambda_i}\Bigr)^{\tfrac12} \leq \sum_{i \leq n} e^{- \frac{t-r}{\eps^2}\lambda_i} \]
which gives
\begin{align*}
    \Bigl|\int_0^t \gamma^{-1} e^{-\frac{t-r}{\eps^2}\gamma} \nabla V(x_r) dr \Bigr|
    &\leq \eps^2 |\gamma^{-1}|_F|\gamma^{-1/2}|_F^2 \max\limits_{r \in [0,t]}|\nabla V(x_r)| ,
\end{align*}
where $|\gamma^{-1/2}|_F^2 = \sum_{i\leq n} 1/\lambda_i$.
Note that $\max\limits_{r \in [0,t]}|\nabla V(x_r)| $ takes a deterministic finite value  since $\nabla V(\cdot)$ and $x_{\cdot}$ are both continuous and hence attain their minimum and maximum on compact intervals.\\
The last term can be bounded using integration by parts as
\begin{align*}
    \Bigl|\sqrt{\eps}\int_0^t   \gamma^{-1}(I_n - e^{-\frac{t-r}{\eps^2} \gamma })\sigma dW_r \Bigr|
    & \leq\sqrt{\eps} \int_0^t \frac{1}{\eps^2}  \bigl|e^{-\frac{t-r}{\eps^2} \gamma }\sigma \bigr|_F \bigl| W_r \bigr|  dr .
\end{align*} 
Since $W_r$ is almost surely continuous, we have that $|W_r| \leq K < \infty$ almost surely so that evaluating the integral $\int_0^t |e^{-\frac{t-r}{\eps^2} \gamma }|_F dr$ yields the almost sure bound 
\begin{align*}
  \Bigl|\sqrt{\eps}\int_0^t  \gamma^{-1}(I_n - e^{-\frac{t-r}{\eps^2} \gamma }) \sigma dW_r \Bigr|
    &\leq   \sqrt{\eps}  |\sigma|_F |\gamma^{-1/2}|_F^2 K .
\end{align*}
Putting everything together gives
\begin{align*}
   \sup_{t \in [0,T]} |X_t - x_t| \leq \sqrt{n} |\gamma^{-1}|_F L_V \int_0^T \sup_{s \in [0,r]}|X_s -x_s| dr + \sqrt{\eps}\, C\,,
\end{align*}
where $C= |\sigma|_F |\gamma^{-1/2}|_F^2 K + \sqrt{\eps}  |\gamma^{-1}|_F \bigl(\sqrt{n}|y| + \eps |\gamma^{-1/2}|_F^2  \max\limits_{r \in [0,T]}|\nabla V(x_r)|\bigr)$ 
which, using Gronwall's inequality, yields the assertion.  The second assertion follows by the modified estimate
\begin{align*}
 \bE\Bigl( \sup\limits_{t \in [0,T]} \Bigl|\sqrt{\eps}\int_0^t  \gamma^{-1}(I_n - e^{-\frac{(t-r)}{\eps^2} \gamma }) \sigma  dW_r \Bigr|^2 \Bigr)
    &\leq  4 \eps T n |\sigma|_F^2 | \gamma^{-1}|_F^2
\end{align*}
for the last term, for which we use Doob's inequality and It\^{o}'s isometry. This, together with the previously derived bounds, Young's inequality, and finally applying Gronwall's inequality gives
\begin{align*}
   \bE\Bigl(\sup_{t \in [0,T]} |X_t - x_t|^2 \Bigr) \leq \eps C_1 e^{C_2} \,,
\end{align*}
where $C_1 =  4  |\gamma^{-1}|_F^2 \bigl[4 Tn |\sigma |_F^2 + \eps^3 |\gamma^{-1/2}|_F^4 \bigl(\max\limits_{t \in [0,T]} |\nabla V(x_t)|\bigr)^2 + \eps  n |y|^2\bigr]$ and $C_2 = 4n |\gamma^{-1}|_F^2 L_V^2 T$.

\bibliographystyle{plain}
\bibliography{langevinHighLow}

\end{document}